\documentclass[11pt]{amsart}
\usepackage{amsfonts, amssymb, amscd, latexsym, graphicx, psfrag, color, float, enumitem}
\usepackage[all]{xy}
\usepackage{mathrsfs}

\usepackage{tikz}

\usepackage[hyphens]{url}

\usepackage{todonotes}
\usepackage{hyperref}

\usepackage[hyphenbreaks]{breakurl}

\usepackage{tkz-euclide}
\usetkzobj{all} 
\usepackage{caption}

\newtheorem{dummy}{dummy}[section]
\newtheorem{lemma}[dummy]{Lemma}
\newtheorem{theorem}[dummy]{Theorem}
\newtheorem{innercustomthm}{Theorem}
\newenvironment{customthm}[1]
{\renewcommand\theinnercustomthm{#1}\innercustomthm}
  {\endinnercustomthm}

\newtheorem{corollary}[dummy]{Corollary}
\newtheorem{proposition}[dummy]{Proposition}
\theoremstyle{definition}
\newtheorem{definition}[dummy]{Definition}
\newtheorem*{definition*}{Definition}
\newtheorem{example}[dummy]{Example}

\newtheorem{remark}[dummy]{Remark}

\newcommand{\bN}{\mathbb{N}}
\newcommand{\bP}{\mathbb{P}}
\newcommand{\bQ}{\mathbb{Q}}

\newcommand{\bZ}{\mathbb{Z}}

\newcommand{\cA}{\mathcal{A}}
\newcommand{\cB}{\mathcal{B}}
\newcommand{\cC}{\mathcal{C}}
\newcommand{\cD}{\mathcal{D}}

\newcommand{\cL}{\mathcal{L}}

\newcommand{\cO}{\mathcal{O}}
\newcommand{\cP}{\mathcal{P}}

\newcommand{\cS}{\mathcal{S}}

\newcommand{\cU}{\mathcal{U}}

\newcommand{\Hom}{\mathrm{Hom}}

\newcommand{\Perf}{\mathrm{Perf}}

\newcommand{\dgCat}{\mathrm{dg}\mathcal{C}\mathrm{at}}

 \newcommand{\twocell}[1]{\ar@{}[#1]^(.30){}="a"^(.70){}="b" \ar@{=>} "a";"b"}
 \newcommand{\ocell}[1]{\ar@{}[#1]^(.30){}="a"^(.70){}="b" \ar@{=} "a";"b"}

\newcommand{\Qcoh}{\mathrm{Qcoh}}

\newcommand\radice[2][\relax]{\hspace{-1.5pt}\sqrt[\uproot{2}#1]{#2}}

\newcommand\Db[1]{\Perf(#1)}

\begin{document}

\author[Scherotzke]{Sarah Scherotzke}
\address{Sarah Scherotzke, 
Universit\'e du Luxembourg\\
Maison du Nombre\\
6, Avenue de la Fonte\\
L-4364 Esch-sur-Alzette\\ Luxembourg}
\email{\href{mailto:sarah.scherotzke@uni.lu}{sarah.scherotzke@uni.lu}}

\author[Sibilla]{Nicol\`o Sibilla}
\address{Nicol\`o Sibilla, SMSAS\\ 
University of Kent\\ 
Canterbury, Kent CT2 7NF\\UK and SISSA\\ Via Bonomea 265 \\34136 Trieste (TS) \\Italy}
\email{\href{mailto:N.Sibilla@kent.ac.uk}{N.Sibilla@kent.ac.uk}}

\author[Talpo]{Mattia Talpo}
\address{Mattia Talpo, Dipartimento di Matematica\\ Universit\`{a} di Pisa \\ Largo Bruno Pontecorvo 5 \\ 56127 Pisa (PI) \\ Italy}
\email{\href{mailto:mattia.talpo@unipi.it}{mattia.talpo@unipi.it}}

\title[Gluing semi-orthogonal decompositions]{Gluing semi-orthogonal decompositions}

\subjclass[2010]{14F05, 19E08}
\keywords{Semi-orthogonal decompositions, root stacks, logarithmic geometry, Kummer flat K-theory}

\begin{abstract} 
We introduce \emph{preordered semi-orthogonal decompositions} (psod-s) of dg-categories. We show that homotopy limits of dg-categories equipped with compatible 
psod-s  carry a natural psod. This gives a way to glue semi-orthogonal decompositions along faithfully flat covers, extending  the main result  
of \cite{bergh2017conservative}. As applications we will 
 construct semi-orthogonal decompositions for root stacks of log pairs $(X,D)$ where $D$ is a (\emph{not necessarily simple}) normal crossing divisor, generalizing results from \cite{ishii2011special}  and   \cite{bergh2016geometricity}.  
Further we will compute the Kummer flat K-theory of general log pairs $(X,D)$, generalizing earlier results of Hagihara and Nizio{\l} in the simple normal crossing case \cite{hagihara}, \cite{Ni1}. 
\end{abstract}

\maketitle


\section{Introduction} 
In this paper we study  conditions under which 
 semi-orthogonal decompositions (\emph{sod-s}) of dg-categories can be glued together to yield global semi-orthogonal decompositions. We formulate our results in terms of  general homotopy limits of dg-categories under appropriate compatibility assumptions on the structure functors. Our main technical  result  is, 
 roughly, that a limit of dg categories equipped with sod-s  and compatible functors between them carries a natural sod (Theorem \ref{psodmain1} in this introduction).

Making use of the homotopy theory of dg categories, the proof of Theorem \ref{psodmain1} is not difficult. However this result has several significant consequences. We will describe them briefly here, while referring the reader to the remainder of the introduction for a fuller summary of the contents of the paper.
\begin{enumerate}
\item As a consequence of Theorem \ref{psodmain1} we recover one of the main results of the interesting recent article \cite{bergh2017conservative}, namely what the authors call \emph{conservative descent}. The proof given in \cite{bergh2017conservative} is framed in the language of  classical triangulated categories, and depends on  rather 
 sophisticated arguments. Leveraging the formalism of $\infty$-categories, however, we can give a very simple proof of this  result. Indeed in section \ref{gpacd} we will show that conservative descent  follows immediately  from our Theorem \ref{psodmain1}. 


\item Root stacks of normal crossing divisors $D \subset X$  have been much studied in algebraic geometry.  
 In particular in \cite{ishii2011special} and \cite{ bergh2016geometricity} it is proved that their derived category  carries a natural  semi-orthogonal 
decomposition. 
These prior results 
assume $D$ to be \emph{simple normal crossing}.  
As an application of Theorem \ref{psodmain1} we drop this assumption. We construct semi-orthogonal decompositions on categories of perfect complexes of root stacks of general  normal crossing divisors $D \subset X$. 
A new feature emerges: whereas  in the simple normal crossing case the semi-orthogonal summands are given by categories of perfect complexes on the strata of $D$, without  the simplicity assumption the summands correspond to perfect complexes on the \emph{normalization} of the strata. 
\end{enumerate}
Although the previous two applications are  of general geometric  import, we were motivated by  log geometry and the theory of parabolic sheaves. Here are two applications of log geometric nature:
\begin{itemize}
\item[(3)] The \emph{infinite root stack}, introduced in \cite{TV}, is an important construction  in log geometry. Using  $(2)$ above, we construct sod-s for infinite root stacks  in the general normal crossing case. 
This improves on a  result from our previous paper \cite[Section 4]{SST2}, where we worked under restrictive assumptions on the ground field (and we used a highly non-trivial invariance of the derived category of infinite root stacks under log blow-ups).  
  By the results of \cite{TV}, we also obtain sod-s on derived categories of parabolic sheaves  of general normal crossing divisors (with rational weights). 
\item[(4)] Hagihara and Nizio{\l}  \cite{hagihara, Ni1} established  an important  structure theorem for Kummer flat K-theory of log schemes with divisorial log structure $(X, D)$, where $D$ is simple normal crossing. They proved that  Kummer flat K-theory splits as an infinite direct sum  labeled by the strata.  As a consequence of $(3)$  we extend their description of Kummer flat K-theory to log pairs $(X, D)$ where 
$D$ is general normal crossing. 
\end{itemize}
 
 \subsubsection*{\textbf{Preordered semi-orthogonal decompositions and gluing}}
 Dg-categories can be viewed as objects inside homotopically enriched categories. In technical terms, we say that dg-categories form a model category or an $\infty$-category. This yields meaningful notions of (homotopy) limits and colimits of dg categories.  This is a key difference with the classical theory of triangulated categories,  where limits and colimits are poorly behaved.  

 In algebraic geometry, descent properties of sheaves  can be encoded via homotopy limits of dg-categories. If $U \to X$ is faithfully flat cover, the category of perfect complexes $\Perf(X)$ can be computed as a limit of the cosimplicial diagram of dg-categories determined by the \v{C}ech nerve of $U \to X$. Limits and colimits of dg-categories arise also in other geometric contexts. For instance, the Fukaya category of exact symplectic manifolds  localizes, 
 and therefore can be calculated as a limit of  Fukaya categories of open patches. 

For this reason, it is important  to have structure theorems that allow us to deduce properties of the limit category from the behaviour of the dg-categories appearing as vertices of the limit. In this paper we prove a result of this type for \emph{semi-orthogonal decompositions} (sod-s).  These are categorified analogues of direct sum decompositions of abelian groups and have long played a key role in algebraic geometry, see \cite{kuznetsov2015semiorthogonal} for a survey of results.  In fact, it is more convenient to work with the slightly more sophisticated concept of \emph{preordered  semi-orthogonal decompositions} (psod-s), where the factors $\cC_w$ of a dg-category $\cC$ are labeled by elements of a preorder $(P, \leq)$. We use the notation $(\cC, P)$ to indicate a dg-category $\cC$ with with a psod indexed by $P$.

We introduce a notion of \emph{ordered} structure on an exact functor $F\colon (\cC_1, P_1) \to (\cC_2, P_2)$: this is the datum of an \emph{order-reflecting} map $\phi_F\colon P_2 \to P_1$ keeping track of the way the psod-s on $\cC_1$ and $\cC_2$ interact via $F$.  
 \begin{customthm}{A}[Theorem \ref{limitpsod2}]
\label{psodmain1}
Assume that for all $i \in I$, 
$\cC_i = (\cC_i, P_i)$ is equipped with a psod, and that for all morphisms  
$ \, 
f\colon i \to j  \, 
$  in  $I,$ 
$\alpha(f)\colon\cC_i \to \cC_j$ is  ordered. Assume additionally that the colimit of indexing preorders 
$
P = \varinjlim_{i \in I} P_i
$
is finite and directed. 
Then the limit category 
$
\cC = \varprojlim_{i \in I} \cC_i
$
carries a  psod with indexing preordered set $P$, 
$ \, 
\cC = \langle \cC_w, w \in P \rangle
\, $, such that for all $w \in P$  we have
$$
\cC_w \simeq \varprojlim_{i \in I} \bigoplus_{z \in \phi_i^{-1}(w)} \cC_{i, z},
 $$ 
 where $\phi_i\colon P_i\to P$ is the natural map.
\end{customthm}

The proof of Theorem \ref{psodmain1} is conceptually clear and  not difficult. It reduces to relatively straightforward manipulations in the $\infty$-category of dg categories. This simplicity is one of the main assets of our approach. As it is often the case, leveraging the power of $\infty$-categories allows for simpler and more conceptual arguments. As an example, we will show how  Theorem \ref{psodmain1} immediately implies an interesting recent result of Bergh and Schn\"urer from \cite{bergh2017conservative}. One of their main theorems is, roughly, a gluing result akin to Theorem \ref{psodmain1}, but limited to the geometric setting, which is called   \emph{conservative descent}. Their approach  is  interesting in itself, but requires rather sophisticated arguments based on the classical theory of triangulated categories. However from a dg point-of-view conservative descent admits a simple proof. Indeed in section  \ref{gpacd} we will show that it can be recovered  as a special case of Theorem \ref{psodmain1}. 

 Since the precise setting of conservative descent is somewhat  complicated, we prefer not to reproduce that result here but refer the reader directly to section \ref{gpacd} in the main text for more details.

  \subsubsection*{\textbf{Perfect complexes on root stacks}} 
Root stacks were first studied systematically by Cadman in \cite{cadman2007using}. 
They carry universal roots of line bundles equipped with a section, and in \cite{cadman2007using} were used to compactify moduli of stable maps. Root stacks have since found  applications  in many different areas of geometry including enumerative geometry, quantum groups \cite{sala2017hall}, the theory of N\'eron models \cite{chiodo2015n}, and more. 
From our perspective, root stacks are an essential tool to probe the geometry of log schemes.

Taking the root stack of a divisor is a fundamental geometric operation akin to blowing-up. In fact these two operations are often combined, as in the notion of \emph{stacky blow-up} proposed by Rydh \cite{rydh2011compactification}. 
From the view-point of the derived category classical blow-ups have a very simple description: the surgery replacing a smooth subscheme with the projectivization of its normal bundle becomes,  at the level of derived categories, the addition of  semi-orthogonal summands to the derived category of the ambient space. It is a natural question, with many important applications, whether this picture extends to stacky blow-ups. Semi-orthogonal decompositions associated to root stacks of normal crossing divisors were studied in \cite{ishii2011special}  and   \cite{bergh2016geometricity}. However these results require 
the normal crossing divisor to be \emph{simple}.

The  assumption of simplicity is artificial from the viewpoint of the underyling geometry. One of the chief goals of this paper is to lift the simplicity assumption and extend these semi-orthogonal decompositions to the general normal crossing case. The definition of the  root stack of a general normal crossing divisor requires some care, but the geometry is clear: the isotropy along the strata of the  divisor keeps track of their codimension.

We formulate a version of our result as Theorem \ref{main1root} below, but we refer the reader to the main text for a sharper and more general statement (see Corollary \ref{sodncver2}). Let $D \subset X$ be a normal crossing divisor. The divisor $D$ determines a  stratification of $X$ where the strata are intersections of local branches of $D$. 
Let $\cS$ be the set of strata and let $M$ be the top codimension of the strata. For every $0 \leq j \leq M$, let 
$\cS_j$ be the set of $j$-codimensional strata.  If $S$ is a stratum in $\cS$ we denote by $\widetilde{S}$ its \emph{normalization}. 
\begin{customthm}{B}[Corollary \ref{sodncver2}]
\label{main1root}
\mbox{}
The dg-category of perfect complexes 
of the $r$-th root stack 
$\radice[r]{(X,D)}$ admits a semi-orthogonal decomposition 
$$
\Perf(\radice[r]{(X,D)}) = \langle   \cA_M,   \ldots, \cA_0   \rangle 
$$
having the following properties:
\begin{enumerate}
\item $\cA_0 \simeq \Perf(X)$,
\item for every $1 \leq j \leq M$, $\cA_j $ decomposes as a direct sum $\cA_j  \simeq \bigoplus_{S \in \cS_j}\cB_S$, and
 \item for every $S \in \cS_j$ the category $\cB_S$ carries a semi-orthogonal decomposition composed of 
  $r\cdot j$ semi-orthogonal factors, which are all equivalent to $\Perf(\widetilde{S})$.
\end{enumerate}
\end{customthm}

\subsubsection*{\textbf{Applications to log geometry: infinite root stacks and Kummer flat K-theory}} 
Log schemes are an enhancement of ordinary schemes which was introduced in the 80's in the context of arithmetic geometry. In recent years log techniques have become  a mainstay of algebraic geometry and mirror symmetry: for instance, log geometry provides the language in which the Gross--Siebert program in mirror symmetry  is formulated \cite{gross2006mirror}.  
In \cite{TV} Talpo and Vistoli explain how to 
to associate to a log scheme its \emph{infinite root stack}, which is a projective limit of usual (i.e. finite-index) root stacks. This assignment gives rise to a faithful functor from log schemes to stacks: log information is converted into stacky information without any data loss. Additionally, the authors prove in \cite{TV}  that the Kummer flat topos of a log scheme $X$ is equivalent to the flat topos of its infinite root stack. 

Using Theorem \ref{main1root}, and passing to the limit on $r$, we obtain an infinite sod on $\Perf(\radice[\infty]{(X,D)})$. This result is stated as Theorem \ref{mainsgncdivinf} in the main text. The passage to the limit actually requires a careful construction of nested sod-s on categories of perfect complexes of root stacks, which was explained in our previous work \cite{SST2},  to which we refer the reader. It follows from results of \cite{TV} that $\Perf(\radice[\infty]{(X,D)})$ is equivalent to the category of parabolic sheaves with rational weights on $(X,D)$: thus our result can be read as the construction of a sod on 
$\mathrm{Par}^{\mathbb{Q}}(X,D)$.

\begin{customthm}{C}[Corollary \ref{ncmotdirectsum}]
\label{ncmotdirectsum11}
Let $(X, D)$ be a log stack given by an algebraic stack $X$ equipped with a normal crossing divisor 
$D.$ 
 Then the \emph{Kummer flat K-theory} spectrum of $X$ decomposes as a direct sum 
$$ 
K_{\mathrm{Kfl}}(X, D) \simeq K(X) \bigoplus \Big ( \bigoplus_{S \in S_D^*} \Big ( \bigoplus_{\chi \in (\mathbb{Q}/\mathbb{Z})_S^*} K(\widetilde S) \Big ) \Big ).
$$

\end{customthm}
We refer the reader to the main text for the definition of all the terms appearing in the formula. Our result extends to the general normal crossing case the structure theorem in Kummer flat K-theory due to Hagihara and Nizio{\l} \cite{hagihara}, \cite{Ni1}. A notable difference from those results is the appearance of the K-theory of the normalization of the strata, rather than of the strata themselves. In fact the statement of Corollary \ref{ncmotdirectsum} in the main text is considerably more general than Theorem \ref{ncmotdirectsum11}: it is not limited to K-theory but holds across all Kummer flat additive invariants.  

\subsection*{Acknowledgments:} We thank Marc Hoyois for a useful exchange about Lemma \ref{limitleft}, and the anonymous referee for useful comments and suggestions. M.T. was partially supported by EPSRC grant EP/R013349/1.  

\subsection*{Conventions}
We will work over an arbitrary ground commutative ring $\kappa$. We use the definition of algebraic stacks  given 
in \cite[Tag 026O]{stacks-project}. In the following all algebraic stacks will be of finite type. All functors between derived categories of sheaves or categories of perfect complexes are implicitly derived.

\section{Preliminaries}
\label{sec:outline}

\subsection{{Categories}} 
\label{sec:dg.categories}

We will work with \emph{dg-categories}, that is, $\kappa$-linear differential graded categories in the sense of  \cite{keller2007differential} and \cite{drinfeld2004dg}. 
If $\cC$ is a dg-category, and $A$ and $B$ are in $\cC,$ we denote by $\Hom_\cC(A, B)$ the Hom-complex between $A$ and $B$. We will be mostly interested in \emph{triangulated dg-categories} which are defined for instance in Section 3 of  \cite{bertrand2011lectures}. The homotopy category of a triangulated dg-category is a triangulated, Karoubi-complete category.

The category of dg-categories and exact functors   carries a model structure, which was studied in \cite{Tab} and \cite{toen2007homotopy},  where weak equivalences are Morita equivalences. A Morita equivalence is a dg-functor $F\colon A \to B$ such that the associated derived functor is an equivalence. 
Localizing this model category  at weak equivalences  yields an $(\infty,1)$-category,  which we denote $\dgCat.$ Our  reference for the theory of $(\infty,1)$-categories is  given by Lurie's work \cite{lurie2009higher, Lu2}. In the rest of the paper we will refer to $(\infty,1)$-categories  simply as $\infty$-categories.   The category $\dgCat$ has a symmetric monoidal structure given by the tensor product of dg-categories,  see \cite{toen2007homotopy}.

We will be interested in taking limits and colimits of diagrams $\alpha\colon I \to \dgCat$, where $I$ is an ordinary category and 
$i$ is a \emph{pseudo-functor}, in the sense for instance of Definition 4.1 of \cite{block2017explicit}. 
All limits and colimits  
are to be understood as  \emph{homotopy} limits and colimits for the Morita model structure: equivalently, they 
are (co)limits in the $(\infty,1)$-categorical sense. Since every  pseudo-functor  $\alpha\colon I \to \dgCat$ can be strictified up to Morita equivalence, and this  does not affect  homotopy (co)limits, the reader can assume that all diagrams $\alpha\colon I \to \dgCat$ in the paper are strict.

Throughout the paper, we will say that a square   of dg-categories 
is \emph{commutative} if there 
is an invertible natural transformation $\alpha\colon KG \Rightarrow HF$  
\begin{equation}
\label{sqdiacomm}
\begin{gathered}
\xymatrix{
\cC_1 \ar[r]^-{F} \ar[d]_-{G} \ar[d] & \cC_2 \ar[d]^-{H} \\
\cC_3 \ar[r]^-{K} 
\twocell{ur}^{\alpha} & \cC_4.}
\end{gathered}
\end{equation}
Whenever only the existence of a natural transformation  making the diagram commute will be needed, and not its explicit definition, we will 
 omit $\alpha$ and the 2-cell notation  from the diagram. We will say that (\ref{sqdiacomm}) \emph{commutes strictly} if $\alpha$ is the indentity natural transformation, and we will sometimes denote this as
 \begin{equation*}
\begin{gathered}
\xymatrix{
\cC_1 \ar[r]^-{F} \ar[d]_-{G} \ar[d] & \cC_2 \ar[d]^-{H} \\
\cC_3 \ar[r]^-{K} 
\ocell{ur} 
& \cC_4.}
\end{gathered}
\end{equation*}
 Let $I$ be a small category. 
Let 
$
\gamma_1, \gamma_2\colon I \to \dgCat
$
be diagrams in $ \dgCat$. For all $i \in I$ set $\cD_i:=\gamma_1(i)$ and $\cC_i:=\gamma_2(i)$.  Let
$$
T\colon \gamma_1 \Rightarrow \gamma_2 \colon I \longrightarrow \dgCat
$$
be a pseudo-natural transformation, given by the following data:
\begin{itemize}[leftmargin=*]
\item for all $i \in I$, a functor $T(i)\colon \cD_i \to \cC_i$,
\item for all $i, j \in I$, and for all maps $i \xrightarrow{a_{i,j}} j$, an invertible natural transformation $\alpha_{i,j}$
\[
\begin{gathered}
\label{diag2}
\xymatrix{
\cD_i \ar[rr]^-{T(i)} \ar[d]_-{\gamma_1(a_{i,j})} \ar[d] && \cC_i \ar[d]^-{\gamma_2(a_{i,j})} \\
\cD_j \ar[rr]^-{T(j)} 
\twocell{urr}^{\alpha_{i,j}} && \cC_j.}
\end{gathered}
\]
\end{itemize}
Denote by $\cD$ and $\cC$ the limits $\varprojlim_{i \in I} \cD_i$ and 
$\varprojlim_{i \in I} \cC_i$ respectively, and let $T$ be the limit of the functors $T(i)$
$$
T= \varprojlim_{i \in I} T(i)\colon \cD \longrightarrow \cC.
$$

\begin{lemma}
\label{limitleft}
Assume the following.
\begin{enumerate}[leftmargin=*]
\item For all $i \in I$, $T(i)$ admits a left   adjoint $T(i)^L$, $T(i)^L \dashv T(i)$.
\item The following \emph{Beck--Chevalley} condition is satisfied: 
for all $i, j \in I$, and for all maps $i \xrightarrow{a_{i,j}} j$,  the canonical  natural transformation  
$$ \alpha_{i,j}^L\colon T(j)^L \circ    \gamma_2(\alpha_{i,j})  \Rightarrow 
  \gamma_1(\alpha_{i,j})  \circ T(i)^L  
$$
induced by $\alpha_{i,j}$  is invertible. 
\end{enumerate}
Then $T^L  = \varprojlim_{i \in I} T(i)^L$ is  the left adjoint of $T$. 
\end{lemma}

\begin{remark}\label{rmk:nat.trans}
Recall that, by definition, $ \alpha^L_{i,j}$ is  given by the composite 
$$
 T(j)^L \circ    \gamma_2(\alpha_{i,j})  \stackrel{(a)}\Rightarrow 
 T(j)^L \circ    \gamma_2(\alpha_{i,j}) \circ T(i) \circ T(i)^L 
 \stackrel{(b)}\Rightarrow
  T(j)^L \circ    T(j) \circ  \gamma_1(\alpha_{i,j}) \circ T(i)^L \stackrel{(c)}\Rightarrow
  \gamma_1(\alpha_{i,j})  \circ T(i)^L
$$
where $(a)$ and $(c)$ are given by the counit of $T(i)^L \dashv T(i)$ and the unit of  $T(j)^L \dashv T(j)$, while $(b)$ is given by $\alpha_{i,j}$. 
\end{remark}
 
\begin{proof}[Proof of Lemma \ref{limitleft}] 
This is a well-known fact in the general 
setting of $\infty$-categories. 
We give  a  proof based on \cite[Appendix D]{bachmann2017norms}.  With small abuse of notation we keep denoting by $I$ also the nerve $\infty$-category of $I$. Via the Grothendieck construction we can write $\gamma_1$ and $\gamma_2$ as Cartesian fibrations over $I^{\mathrm{op}}$. Then  $T$ yields a morphism of cartesian fibrations over $I^{\mathrm{op}}$
$$
\xymatrix{
\int_I \gamma_1 \ar[rd] \ar[rr]^T && \int_I \gamma_2 \ar[ld] \\
& I^{\mathrm{op}} &
}
$$ 
By assumption (1), 
$T$ has a relative left adjoint: this follows from Lemma D.3 
of \cite{bachmann2017norms}. Now by Lemma D.6 of \cite{bachmann2017norms}  relative adjunctions induce adjunctions between $\infty$-categories of sections. Further, if the relative left adjoint preserves cartesian edges (which is the case by assumption (2)
), this restricts to an adjunction between the full subcategories of cartesian sections: this gives the desired adjunction 
$T^L \dashv T$ and concludes the proof. 
\end{proof}
We also state the analogue of Lemma \ref{limitleft} for right adjoints. 
\begin{lemma}
\label{limitright}
Assume the following.
\begin{enumerate}[leftmargin=*]
\item For all $i \in I$, $T(i)$ admits a right adjoint $T(i)^R$, $T(i) \dashv T(i)^R$.
\item The following \emph{Beck--Chevalley} condition is satisfied: 
for all $i, j \in I$, and for all maps $i \xrightarrow{a_{i,j}} j$,  
the canonical  natural transformation  
$$\alpha^R_{i,j}\colon  \gamma_1(\alpha_{i,j})  \circ T(i)^R \Rightarrow 
  T(j)^R \circ    \gamma_2(\alpha_{i,j}) $$
induced by $\alpha_{i,j}$  is invertible. 
\end{enumerate}
Then $T^R = \varprojlim_{i \in I} T(i)^R$ is  the right adjoint of $T$. 
\end{lemma}

\subsubsection{Categories of sheaves}
If $X$ is an algebraic  stack, we denote by  $\Qcoh(X)$ the 
 triangulated dg-category of quasi-coherent sheaves on $X$.  
The tensor product of quasi-coherent sheaves equips  $\Qcoh(X)$ with a symmetric monoidal structure, and $\Db{X}$, the dg-category of perfect complexes, is defined as the full subcategory of dualizable objects (see \cite{BFN}). By \cite[Theorem 1.3.4]{G},  the dg-category of quasi-coherent sheaves $\Qcoh(-)$ satisfies faithfully flat descent: given a  faithfully flat cover $Y \to X,$  if $Y^\bullet$ is the semi-simplicial object given by the \v{C}ech nerve of $Y \to X$, then $ \varprojlim \Qcoh(Y^\bullet) \simeq \Qcoh(X).$
 Passing to dualizable objects on both sides, we obtain that $\Db{-}$ also satisfies faithfully flat descent.

\subsubsection{Exact sequences}
 Let $\cC$ be a triangulated dg-category.  We say that two objects $A$ and $A'$ are \emph{equivalent} if there is a degree $0$ map $A\to A'$ that becomes an isomorphism in the homotopy category of $\cC$.  If $\iota\colon \cC' \to \cC$ is a fully faithful functor, we often view $\cC'$  as a subcategory of $\cC$ and identify $\iota$ with an inclusion.  Accordingly, we will usually denote the image under  $\iota$ of an object $A$ of $\cC'$  simply by $A$ rather than $\iota(A).$ 
We will always assume that subcategories are closed under equivalence. That is, if $\cC'$ is a full subcategory of $\cC,$ $A$ is an object of $\cC',$ and $A'$ is an object of $\cC$ which is equivalent to $A,$ we will always assume that $A'$ lies in $\cC'$ as well.

Recall that if $\cD$ is a full subcategory of $\cC,$ $(\cD)^\bot$ denotes the \emph{right orthogonal} of $\cD$, i.e. the full subcategory of $\cC$ consisting of the objects $B$ such that the Hom-complex $\Hom_\cC(A,B)$ is acyclic for every object $A \in \cD$. 
Let 
$\{\cC_1, \ldots, \cC_n\}$ be a finite collection of triangulated subcategories of $\cC$ such that, for all $1 \leq i < j \leq n,$ $\, \cC_{i} \subseteq (\cC_{j})^\bot.$ 
Then we denote  by 
$
\langle \cC_1 \, , \ldots \,, \cC_n \rangle
$
 the smallest 
 triangulated subcategory 
of $\cC$ containing all the  subcategories $\cC_i$.  
An \emph{exact sequence} of triangulated dg-categories is a sequence  
\begin{equation}
\label{dgverdier}
\cA \stackrel{F} \longrightarrow \cB \stackrel{G} \longrightarrow \cC \end{equation}
which is both a fiber and a cofiber sequence in $\dgCat.$  This  is an analogue of classical \emph{Verdier localizations} of triangulated categories in the dg setting. Exact sequences of dg-categories are detected at the homotopy level: it can be shown that (\ref{dgverdier}) is an exact sequence  if and only if 
the sequence of homotopy categories
$$
\mathrm{Ho}(\cA) \xrightarrow{\mathrm{Ho}(F)}  
\mathrm{Ho}(\cB) 
\xrightarrow{\mathrm{Ho}(G)} 
\mathrm{Ho}(\cC)
$$  
is a classical Verdier localization of triangulated categories.

The functor $F$ admits a right adjoint $F^R$ exactly if $G$ admits a right adjoint  $G^R$, and similarly for left adjoints, see e.g. \cite[Proposition 4.9.1]{krause2010localization}.  
If $F$ (or equivalently $G$) admits a right adjoint we say that (\ref{dgverdier}) is a \emph{split  exact sequence}. In this case the functor $G^R$ is fully faithful 
and we have that  
$
\cB=\langle G^R(\cC) \, , \, \cA \rangle. 
$
As we  indicated earlier, since $G^R$ is fully faithful we will drop it from our notations  whenever this is not likely to create confusion: thus we will denote  
$G^R(\cC)$ simply by $\cC,$  and 
write 
$
\cB=\langle \cC\, , \,  \cA \rangle. 
$
 
 \subsection{Root stacks} 
 For the convenience of the reader, we include a brief reminder about root stacks of Cartier divisors in an algebraic stack. More details can be found in \cite[Section 2.1]{SST2} and references therein.
 
 Let $X$ be a scheme. Then a Cartier divisor $D\subset X$ with ideal sheaf $I\subset \cO_X$ is said to be \emph{simple} (or \emph{strict}) \emph{normal crossings} if for every $x\in D$ the local ring $\cO_{X,x}$ is regular, and there exist a regular sequence  $f_1,\hdots, f_n\in \cO_{X,x}$ such that $I_x=\langle f_1,\hdots, f_k\rangle\subset \cO_{X,x}$ for some $k\leq n$. Moreover, $D$ is said to be \emph{normal crossings} if \'etale locally on $X$ it is  simple normal crossings. These notions are naturally extended to algebraic stacks by working on an atlas.
  
Given an algebraic stack $X$ with a normal crossings divisor $D\subset X$ one can form a root stack $\radice[r]{(X,D)}$ for every $r\in \bN$. If $D$ is simple normal crossings, this has a simple functorial description as the stack parametrizing tuples $(\cL_1,s_1),\hdots, (\cL_k,  s_k)$ of line bundles with global sections, with isomorphisms $(\cL_i,s_i)^{\otimes r}\cong (\cO(D_i),s_{D_i})$, where $D_i$ are the irreducible components of $D$ and $s_{D_i}$ is the canonical section of the line bundle $\cO(D_i)$, whose zero locus is $D_i$. Passing to $\radice[r]{(X,D)}$ has the effect of attaching a stabilizer $\mu_r^k$ to points in the intersection of exactly $k$ irreducible components of $D$.

When $D$ is only normal crossings, this description is not correct, because it doesn't distinguish the branches of $D$ at points where an irreducible components self-intersects (as for example in the node of an irreducible nodal plane cubic). In this case we have to use the notion of a root stack of a logarithmic scheme (see \cite{borne-vistoli,talpo2014moduli}). In this particular case, we can think about $\radice[r]{(X,D)}$ as being the gluing of the $r$-th root stacks $\radice[r]{(U,D|_U)}$, where $\{U\to X\}$ is an \'etale cover of $X$ where $D$ becomes simple normal crossings.

The root stacks $\radice[r]{(X,D)}$ form an inverse system. Indeed, if $r\mid r'$ there is a natural projection $\radice[r']{(X,D)}\to \radice[r]{(X,D)}$. The inverse limit $\radice[\infty]{(X,D)}=\varprojlim_r \radice[r]{(X,D)}$ is the \emph{infinite root stack} \cite{TV} of $(X,D)$. It is a pro-algebraic stack that embodies the ``logarithmic geometry'' of the pair $(X,D)$ in its stacky structure, and it is an algebraic analogue of the so-called ``Kato-Nakayama space'' \cite{CSST, TaV}. In particular, quasi-coherent sheaves on it correspond exactly to parabolic sheaves \cite{borne-vistoli} on the pair $(X,D)$, and finitely presented sheaves can be identified with finitely presented sheaves on the Kummer-flat site of $(X,D)$ (and also on the Kummer-\'etale site, if the base ring has characteristic $0$).

\section{Preordered semi-orthogonal decompositions}
\label{psodsec}

 \label{gluing psods}
In this section we  introduce \emph{preordered semi-orthogonal decompositions} (psod-s). We will study limits of dg-categories equipped with compatible psod-s. This concept was also considered in \cite{bergh2016geometricity}. Once 
the set-up is in place the proof of the main result (Theorem \ref{limitpsod2}) is not difficult,  relaying as it does on general properties of limits  in the $\infty$-category of dg-categories. As an application of our theory, we will obtain gluing results for semi-orthogonal decompositions along appropriate faithfully flat covers. 
In the setting of classical triangulated categories a related result, called \emph{conservative descent}, has been obtained in the recent paper  \cite{bergh2017conservative}. The advantage of our set-up is that conservative descent follows in a straightforward way from  the general formalism, and thus the proof  that we will give is considerably easier than the one contained in \cite{bergh2017conservative}.

 We will follow closely the account of psod-s contained in the previous paper of the authors  \cite{SST2}. We refer the reader to  \cite{kuznetsov2015semiorthogonal} for a comprehensive survey of the classical theory of semi-orthogonal decompositions.

Let $\cC$ be a triangulated dg-category, and let $P$ be a preordered set. Consider a collection 
of full triangulated  subcategories $\iota_x\colon \cC_x  
\longrightarrow \cC$ indexed by $x \in P$.

\begin{definition}\label{def psod} \mbox{ }
\begin{itemize} 
\item The subcategories $\cC_x$ form a    \emph{preordered semi-orthogonal decomposition (psod) of type $P$}  if they satisfy the following three properties.
\begin{enumerate}
\item For all $x \in P,$ $\cC_x$ is a non-zero  \emph{admissible} subcategory: that is, the embedding  $
\iota_x $ admits a right adjoint and a left adjoint, which we denote by
$$
r_x\colon \cC \longrightarrow \cC_x \quad \text{and} \quad 
l_x\colon \cC \longrightarrow \cC_x .
$$
\item If $y<_P x$, i.e. $y \leq_P  x,$ and $ x \not =y,$ then $\cC_y \subseteq \cC_x^{\bot}.$ 
\item $\cC$ is the smallest stable subcategory of $\cC$ containing $\cC_x$ for all $x \in P.$
\end{enumerate}
\item We say that the subcategories $\cC_x$ form a    \emph{pre-psod  of type $P$}  if they satisfy only properties $(1)$ and $(2).$ 
\end{itemize}
\end{definition}

Note that from condition $(2)$ it follows that if we have both $y<_P x$ and $x<_P y$, then $$\langle \cC_x,\cC_y\rangle =\langle \cC_y,\cC_x\rangle \simeq \cC_x\oplus \cC_y$$

\begin{definition}
If $\cC$ is equipped with a psod of type $P$,  we write
$
\cC = \langle \, \cC_x, \, x \in P \, \rangle. 
$
 We sometimes denote the category 
$\cC$ by 
$(\cC, P)$ to make explicit the role of the indexing preordered set $P.$  
\end{definition}

\begin{remark}
If $\cC$ is the zero category, then it carries a psod indexed by the empty preordered set. 
\end{remark}

We will be interested in gluing psod-s along limits of dg-categories. This requires introducing an appropriate notion of exact functor compatible with psod-s. We do so after introducing some preliminary concepts. 

\begin{definition}
Let $P$ and $Q$ be preordered sets. We say that a map of sets 
$
\phi\colon Q \to P
$
is \emph{order-reflecting} if for all 
$x,y$ in $Q$ we have $\phi(x) \leq_P  \phi(y) \; \Longrightarrow \; x \leq_Q y.$
\end{definition}

We denote by $\mathrm{PSets}^{\mathrm{refl}}$  the category of preordered sets and order-reflecting maps between them. Let us 
summarize some of its basic properties. Note first that the forgetful functor to sets 
$$
U\colon \mathrm{PSets}^{\mathrm{refl}} 
\longrightarrow \mathrm{Sets}
$$
admits a left and a right adjoint, $ \, U^L \dashv U \dashv U^R \, $.  
  The functor $U^L$ sends a set $S$ to the preordered set   
$(S, \leq )$ such that $x \leq   y$ for all  $x, y \in S,$ with the obvious definition on morphisms. We call $U^L(S)$ the \emph{complete preorder} on the set $S.$    
 The functor $U^R$ sends a set $S$ to the preordered set   
$(S, \leq )$ such that $x \leq   y,$ if and only if $x=y,$ with the obvious definition on morphisms. We call $U^R(S)$ the \emph{discrete preorder} on the set $S.$
 
The category $\mathrm{PSets}^{\mathrm{refl}}$ admits all small limits and colimits. Since $U$ has a right and a left adjoint, they are computed by the  underlying  set-theoretic limits and colimits.  
In particular, the coproduct of a collection of 
partially ordered sets  $\{(P_i, \leq_{P_i} )\}_{i \in I}$ is given by the disjoint union  
$
 P=\coprod_{i \in I} P_i 
$ equipped with a preorder $\leq_P$  defined as follows: let $i, j \in I$ and let $x$ be in  
$P_i$ and $y$  in $P_j,$  then 
\begin{itemize}
\item if $i=j,$  we have $x \leq_P  y$ if and only if $x \leq_{P_i} y$;  
 \item if $i \neq j,$  we have $x \leq_P y$.
 \end{itemize}
Let us describe next the push-out of  preordered sets  in $\mathrm{PSets}^{\mathrm{refl}},$
$$
\xymatrix{
(P, \leq_P) &  \ar_{\pi_1}[l]   (P_1, \leq_{P_1})  \\
(P_2, \leq_{P_2}) \ar^{\pi_2}[u]& \ar[l]^-{\pi_4}  \ar[u]_-{\pi_3} (P_3, \leq_{P_3}).}
$$
The  set $P$ is the push-out of the underlying sets,   
and is equipped with a preorder $\leq_P$  defined as follows: let $z$ and $z'$ be in 
$P,$ then $z \leq_P z' $ if and only if    \begin{itemize}
  \item for all pairs $x, x' \in P_1$ such that $\pi_1(x)=z$ and  $\pi_1(x')=z'$,  we have $ x \leq_{P_1} x',$ and 
   \item for all pairs $x, x' \in P_2$ such that $\pi_2(x)=z$ and  $\pi_2(x')=z'$,  we have $ x \leq_{P_2} x'.$
  \end{itemize}

\begin{remark}
\label{discdisc}
Let $\phi\colon (P, \leq_P) \to (Q, \leq_Q)$ be an order-reflecting map. Then, for all $x \in Q,$ the preordered set $\phi^{-1}(x) \subseteq (P, \leq_P)$ is equipped with the complete preorder. Indeed, since $\phi$ is order-reflecting, for all $y$ and 
$y'$ in $\phi^{-1}(x)$ we must have that $y \leq_P y'$ and $y' \leq_P y.$ 
\end{remark}

Next, we introduce a notion of compatible functor between categories equipped with a psod. 

\begin{definition}
\label{fun psod}
Let $(\cC,P)$ and $(\cD,Q)$ be triangulated dg-categories equipped with a pre-psod 
$$
\cC= \langle \, \cC_x, \, x \in P \, \rangle \quad \text{and} \quad  \cD = \langle \, \cD_y, \, y \in Q \, \rangle  
$$
and let  $ \, 
F\colon \cC \rightarrow \cD
\, $ be an exact functor. 
A structure of \emph{ordered functor} on $F$ is the datum of a function $\phi_F\colon Q\to P$  satisfying the following properties: 
\begin{itemize}[leftmargin=*] 
\item The function $\phi_F$ is an 
order-reflecting map,  
and for all $x$ in $P$ there is a \emph{strictly commutative} square
\[
\begin{gathered}
\xymatrix{
\cC_x \ar[r]^-{\iota_x} 
\ar[d]_-{F|_{\cC_x}}  & \cC \ar[d]^-F \\
 \langle  \cD_y \, , \, y \in \phi_F^{-1}( x) \rangle
\ar[r]^-\simeq  \ocell{ur}      \ar[r] & \cD.}
\end{gathered}
\]
\item For all $x \in P,$ set   
$
\, \, r_{\phi_F^{-1}(x)} := \bigoplus_{y \in \phi_F^{-1}( x)} r_y \, \, ,$  and 
$
\, \, l_{\phi_F^{-1}(x)} := \bigoplus_{y \in \phi_F^{-1}( x)} l_y \, \, ,$
$$
r_{\phi_F^{-1}(x)} \colon \cD \longrightarrow \bigoplus_{y \in   \phi_F^{-1}(x)} \cD_y, \quad 
l_{\phi_F^{-1}(x)}\colon \cD \longrightarrow \bigoplus_{y \in   \phi_F^{-1}(x)} \cD_y. 
$$
Then  the following \emph{Beck--Chevalley  condition} holds: there are commutative diagrams
$$
\xymatrix{
\cC_x \ar[d]_-{F|_{\cC_x}}  && \cC \ar[d]^-F \ar[ll]_-{r_x}   
\twocell{dll}^{\alpha_x^R}
&& \cC_x \ar[d]_-{F|_{\cC_x}}  && \cC \ar[d]^-F \ar[ll]_-{l_x}\\
\bigoplus_{y \in   \phi_F^{-1}(x)} \cD_y && \cD \ar[ll]^ -{r_{\phi_F^{-1}(x)}} && \bigoplus_{y \in   \phi_F^{-1}(x)} \cD_y && \cD \ar[ll]^ -{l_{\phi_F^{-1}(x)}} \twocell{ull}^{\alpha_x^L}
}
$$
where $\alpha_x^R$ and $\alpha_x^L$ are defined as in Remark \ref{rmk:nat.trans}.
\end{itemize}
\end{definition}

\begin{remark}
Being \emph{ordered} is a structure on a functor $F\colon (\cC, P) \to (\cD, Q)$ and \emph{not a property}: there might be more than one function $\phi_F$ satisfying the properties of Definition \ref{fun psod}.  
\end{remark}

\begin{remark}
\label{rem fun psod}
As we noted in Remark \ref{discdisc}, $\phi_F^{-1}(x)$ is equipped with the complete preorder.   This yields a canonical  equivalence 
$
\bigoplus_{y \in   \phi_F^{-1}(x)} \cD_y \simeq   \langle  \cD_y \, , \, y \in \phi_F^{-1}( x) \rangle
$
which we assumed implicitly in the formulation of property  $(2)$ of Definition \ref{fun psod}. In particular,  
$r_{\phi_F^{-1}(x)}$ and $l_{\phi_F^{-1}(x)}$ are 
the right and left adjoints of the fully faithful 
functor $\iota_{\phi_F^{-1}(x)}:= \oplus_{y \in \phi_F^{-1}} \iota_y\colon \langle  \cD_y \, , \, y \in \phi_F^{-1}( x) \rangle 
\to \cD $.
\end{remark}

Let $I$ be a small category, and consider a diagram
$
\alpha\colon I \to \dgCat
$. For all $i \in I$ set $\cC_i:=\alpha(i)$. 
Assume  that:
\begin{enumerate}
\item for all $i \in I$, the category $\cC_i$ is equipped with a pre-psod $\cC_i=(\cC_i, P_i)$, and
\item for all arrows $f\colon i \to j$ in $I$, the functor $\alpha(f)\colon (\cC_i, P_i) \to (\cC_j, P_j)$ is equipped with a ordered structure $\phi_{\alpha(f)}$.
\end{enumerate}
Assume additionally that the assignments  
$$
i \in I \, \mapsto \, P_i \quad \text{and}  \quad (f\colon i \to j) \in I \, \mapsto \, \phi_{\alpha(f)}\colon P_j \to P_i
$$
yield a well-defined functor $I^{op} \to \mathrm{PSets}^{\mathrm{refl}}$.  
Let $P$ be the colimit of this diagram, and for all $i \in I$ let $\phi_i\colon P_i \to P$ be the structure maps.

\begin{proposition}
\label{limit of psod} 
The limit category 
$
\cC = \varprojlim_{i \in I} \cC_i
$
carries a pre-psod with indexing preordered set
$
P = \varinjlim_{i \in I} P_i 
$ such that, for all $w \in P$ the subcategory $ \cC_w \subset \cC$ is given by 
$$
  \cC_w \simeq \varprojlim_{i \in I} \bigoplus_{z \in \phi_i^{-1}(w)} \cC_{i, z}.
 $$
\end{proposition}
\begin{proof}
Every limit can be expressed in terms of products and fiber products. Thus it is sufficient to show the statement for these two classes of limits. Let us consider the case of products first. The product of dg-categories $\{\cC_i\}_{i \in I}$ is the limit of the zero functors
$
(\cC_i, P_i) \longrightarrow (0, \varnothing). 
$ 
Thus, we need to show that the product category 
$\cC = \prod_{i \in I}\cC_i$
carries a pre-psod indexed by 
$
P = \coprod_{i \in I} P_i.  
$
If $x$ is in $P,$ there is a $j \in I$ such that $x$ is in $P_j \, ,$ and   
we denote  
$
\iota_x\colon \cC_x \longrightarrow \cC 
$
the subcategory of $\cC$ given by 
$$
\cC_x := (\cC_j)_x  \times \langle 0 \rangle  \longrightarrow \cC_j \times \prod_{i \in I, \, i \neq j}\cC_i \stackrel{\simeq} \longrightarrow \cC. 
$$
It is immediate to verify that the collection of subcategories 
$
\cC_x$ for $ x \in P$
satisfies properties $(1)$ and $(2)$ from Definition 
\ref{def psod}, and thus that it is a pre-psod.

Let us check next that the statement holds for fiber products.  Let 
$$
\xymatrix{
\cC \ar[d]_-H \ar[r]^-K & \cC_1 \ar[d]^-F \\
\cC_2 \ar[r]^-{G} & \cC_3
}
$$
be a fiber product of triangulated dg-categories, such that $\cC_1,$ $\cC_2$ and $\cC_3$ are equipped with a pre-psod
$$
(\cC_1, P_1) = \langle \,  \cC_{1,x}, \, x \in P_1 \, \rangle, \quad 
(\cC_2, P_2) = \langle \, \cC_{2,y}, \, y \in P_2 \, \rangle, 
\quad 
(\cC_3, P_3) = \langle \, \cC_{3,z}, \, z \in P_3 \, \rangle,
$$
and $G$ and $F$ are ordered functors. 
Let $P$ be the pushout of the $P_i$, and denote 
$$
\phi_K\colon P_1 \to P, \quad \phi_H\colon P_2 \to P
$$
the corresponding order-reflecting maps. We also set 
$\phi_{FK}:=\phi_K \circ \phi_F$ and 
$\phi_{GH}:=\phi_H \circ \phi_G$. Note that, since $P$ is a push-out, $\phi_{FK} = \phi_{GH}$.

For all 
$w$ in $P$ we set 
$$
\cC_{w} := \langle \cC_{1,x} \, , \, x \in \phi_K^{-1}(w)  \rangle \times_{\cC_3}  \langle \cC_{2,y} \, , \, y \in \phi_H^{-1}(w) \rangle. 
$$
Since $\langle \cC_{1,x} \, , \, x \in \phi_K^{-1}(w)  \rangle$ and $\langle \cC_{2,y} \, , \, y \in \phi_H^{-1}(w) \rangle$ are full subcategories of $\cC_1$ and $\cC_2$, we have that 
$\cC_w$ is a full subcategory of $\cC = \cC_1 \times_{\cC_3} \cC_2.$ Note that we can write $\cC_w$ equivalently as the   fiber product $$ \langle \cC_{1,x} \, , \, x \in \phi_K^{-1}(w)  \rangle \times_{\langle \cC_{3, z} \, , \, z \in \phi_{GH}^{-1}(w) \rangle}  \langle \cC_{2,y} \, , \, y \in \phi_H^{-1}(w) \rangle,  
$$  
since $\langle \cC_{3, z} \, , \, z \in \phi_{GH}^{-1}(w) \rangle $ is a full triangulated subcategory of $\cC_3$, and the functors $$\langle \cC_{1,x} \, , \, x \in \phi_K^{-1}(w)  \rangle \to \cC_3 \leftarrow \langle \cC_{2,y} \, , \, y \in \phi_H^{-1}(w) \rangle$$ factor through it. 

We will show that the collection of subcategories of $\cC$  given by 
$
\langle \, \cC_{w}, \, w \in P \, \rangle
$ 
satisfies the properties of a pre-psod of type $P.$ Property $(1)$ follows from Lemma \ref{limitleft}. Thus we are reduced to checking property $(2)$. In order to do this, it is useful to use an explicit model for the fiber product 
of dg-categories, which can be found for instance in 
\cite[Appendix IV]{drinfeld2004dg}. The category $\cC_1 \times_{\cC_3} \cC_2$ 
has
\begin{itemize}
\item as objects, triples $(A_1, A_2, u\colon F(A_1) \to G(A_2)),$ 
where $A_1$ is in $\cC_1,$ $A_2$ is in $\cC_2,$ and $u$ is an equivalence,  
\item while the Hom-complex 
$
\mathrm{Hom}_{\cC}((A_1, A_2, u), (A'_1, A'_2, u'))
$
is given by the cone of the  map 
\begin{equation}
\label{hom-cmplx}
\Hom_{\cC_1} (A_1, A'_1) \oplus \Hom_{\cC_2} (A_2, A'_2) \xrightarrow{u'F- Gu}   \Hom_{\cC_3} (F(A_1), G(A'_2)).
\end{equation}
\end{itemize}
 If $w < w'$ are distinct elements of $P,$ we have to show that $\cC_{w} \subseteq \cC_{w'}^{\bot}.$ That is, we need to prove that if $(A_1, A_2, u)$ is in $\cC_w$ and $(A'_1, A'_2, u')$ is  in $\cC_{w'},$ then 
$
\mathrm{Hom}_{\cC}((A_1', A_2', u'), (A_1, A_2, u)) \simeq 0. 
$
 This however  follows immediately by the calculation of the Hom-complexes in $\cC$ given by (\ref{hom-cmplx}). Indeed, since we have inclusions  
 $$
\bigoplus_{x \in \phi_K^{-1}(w) } \cC_{1, x} \subseteq \Big ( \bigoplus_{x \in \phi_K^{-1}(w') } \cC_{1, x} \Big )^{\bot} \quad , \quad \bigoplus_{y \in \phi_H^{-1}(w) } \cC_{2, y} \subseteq \Big ( \bigoplus_{y \in \phi_H^{-1}(w') } \cC_{2, y} \Big )^{\bot}
 $$
 the source of the morphism of complexes (\ref{hom-cmplx}) vanishes.  
Further, we have that  
$$ F(A'_1) \in \langle  \cC_{3,z} \, , \, z \in \phi_{FK}^{-1}(w')   \rangle \quad \text{and} \quad G(A_2) \in \langle  \cC_{3,z} \, , \, z \in \phi_{GH}^{-1}(w)   \rangle,$$ and as $\phi_{GH}=\phi_{FK}$ is order-reflecting we have that 
$$
\langle  \cC_{3,z} \, , \, z \in \phi_{GH}^{-1}(w)   \rangle \subseteq \Big (\langle  \cC_{3,z} \, , \, z \in \phi_{FK}^{-1}(w')   \rangle \Big )^{\bot}.
$$ 
Hence also the target of the morphism of complexes (\ref{hom-cmplx}) vanishes and thus its cone is zero, and this concludes the proof. 
\end{proof}

We will be interested in calculating limits of categories equipped with actual psod-s, rather than just pre-psod-s. However, in  general  we cannot conclude that the limit category $\cC$ will carry a psod, as the admissible subcategories constructed in the proof of Proposition \ref{limit of psod}  might fail to generate $\cC.$ We clarify this point via an example in Example \ref{rem inf prod} below. 
Then, in Theorem \ref{limitpsod2}, we give sufficient conditions ensuring that the limit category will carry an actual psod.

\begin{example}
\label{rem inf prod}
Let $\{\cC_n \}_{n \in \mathbb{N}}$ be a collection of triangulated dg-categories. We can equip them with a psod indexed by the trivial preorder 
$
P_n = \{*\}$ for all $n \in \mathbb{N}.$  
Then Proposition  \ref{limit of psod} yields a pre-psod on $\cD = \prod_{n \in \mathbb{N}} \cC_n$ indexed by the set $\mathbb{N}$ equipped with 
the discrete preorder: the subcategories of $\cC$ forming this pre-psod are given by 
$$
\cD_n:=\cC_n  \times \langle 0 \rangle  \longrightarrow \cC_n \times \prod_{m \in \mathbb{N}, \, m \neq n}\cC_m \stackrel{\simeq} \longrightarrow \cC. 
$$
It is easy to see that the collection $\{\cD_n\}_{n \in \mathbb{N}}$ fails to be a psod. Indeed the category spanned by the 
$\cD_n$ is
$
\langle \cD_n \, , \, n \in \mathbb{N} \rangle \simeq \bigoplus_{n \in \mathbb{N}} \cD_n,
$
which is strictly contained in $\cC$. 
\end{example}

Let  
$(\mathbb{N}, \leq)$ be the set of natural numbers equipped with their usual ordering. 
\begin{definition}
\label{defdirected}
A preordered set $(P, \leq_P)$ is \emph{directed} if there exist an order-reflecting map 
$$(P, \leq_P) \to (\mathbb{N}, \leq).$$ 
\end{definition}
\begin{remark}
\label{rem numbering}
Note that if $(P, \leq_P)$ is a directed finite preorder, then we can  number   its elements 
$
\{p_0, \ldots, p_m\}
$
by natural numbers in such a way that, if $0 \leq n < n' \leq m$, then $p_n <_P p_{n'}.$ 
\end{remark}
In the statement of Theorem \ref{limitpsod2} below we use the same notations as in Proposition \ref{limit of psod}, and make the same assumptions that were made there.   
\begin{theorem}
\label{limitpsod2}
Assume that for all $i \in I$, 
$\cC_i = (\cC_i, P_i)$ is equipped with a psod.   
Assume also that the colimit of indexing preorders 
$
P = \varinjlim_{i \in I} P_i
$
is finite and directed. 
Then the limit category 
$
\cC = \varprojlim_{i \in I} \cC_i
$
carries a  psod with indexing preordered set $P$, 
$ \, 
\cC = \langle \cC_w, w \in P \rangle
\, $, such that for all $w \in P$  
$$
\cC_w \simeq \varprojlim_{i \in I} \bigoplus_{z \in \phi_i^{-1}(w)} \cC_{i, z}.
 $$ 
\end{theorem}
\begin{proof}
We use the same notations we  introduced in the proof of Proposition \ref{limit of psod}. In particular, Proposition \ref{limit of psod} yields a collection of subcategories 
$
\cC_w$ for $ w \in P 
$
satisfying properties $(1)$ and $(2)$ of a psod in Definition \ref{def psod}. We only need to show that these subcategories generate $\cC.$ 
 
 We denote by
 $
 \alpha_i\colon \cC \longrightarrow \cC_i
 $
 the universal functors from the limit category. Note that by construction, these are ordered functors.
 Using the directedness of $P$ we can choose a numbering of  its elements 
 $
 \{w_0, \ldots, w_m\}
 $
 having the property discussed in Remark \ref{rem numbering}. Recall also that for $w\in P$ we denote by $r_w\colon \cC\to \cC_w$ the right adjoint of the inclusion $\iota_w\colon \cC_w\to \cC$. 
Let us pick  a non-zero object $A \in \cC$, and show that it belongs to the subcategory $ \langle \cC_w \, , \, w \in P \rangle$. 
 Since $\cC_{w_m}$ is right-admissible, there is a triangle 
 $$
 \iota_{w_m}r_{w_m}A \longrightarrow A \longrightarrow A_{\cC_{w_m}^\bot}  
 $$
 where $A_{\cC_{w_m}^\bot}$ has the property that $r_{w_m}(A_{\cC_{w_m}^\bot}) \simeq 0.$ Now set 
 $A_1:=A_{\cC_{w_m}^\bot},$ and consider   the analogous triangle for $A_1$, using $\cC_{w_{m-1}}$ instead of $\cC_{w_m}$
 $$
 \iota_{w_{m-1}}r_{w_{m-1}}A_1 \longrightarrow A_1 \longrightarrow A_{\cC_{w_{m-1}}^\bot}. 
 $$
Note that $r_{w_{m-1}}(A_{\cC_{w_{m-1}}^\bot})\simeq 0,$ and also 
  $r_{w_m}(A_{\cC_{w_{m-1}}^\bot})\simeq 0$ as 
 \begin{itemize}
 \item $
  r_{w_m}( \iota_{w_{m-1}}r_{w_{m-1}}A_1)\simeq 0, 
  $
    because $\iota_{w_{m-1}}r_{w_{m-1}}A_1 \in \cC_{w_{m-1}} \subseteq \cC_{w_m}^{\bot}$, 
  \item and $r_{w_m}(A_1)\simeq 0,$ because by construction $A_1 \in \cC_{w_m}^{\bot}.$ 
  \end{itemize} 
Next we set $A_2:=A_{\cC_{w_{m-1}}^\bot},$ and we can iterate the  construction above, this time with respect to $\cC_{w_{m-2}}$.

  Since $P$ is finite, in this way we construct inductively 
 an object $A_{m+1} \in \cC$ having the property that 
 $
 r_{w_i}(A_{m+1})\simeq 0 $ \text{for all} $ 0 \leq i \leq m. 
 $
Since the functors $\alpha_i$ are ordered, this implies that $r_{x}\alpha_i(A_{m+1})\simeq 0$ in $\cC_{i,{x}}$ for every $i$ and every $x \in P_i$. As each of the categories $(\cC_i, P_i)$ is generated by the subcategories making up their psod-s,  this implies that 
$\alpha_i(A_{m+1}) \simeq 0$ in $\cC_i$ for all $i$ in $I.$ Thus $A_{m+1} \simeq 0.$  As a consequence $A$ can be realized as an iterated cone of objects lying in the subcategories $\cC_w$  and therefore it lies  in $\langle \cC_w \, , \, w \in P \rangle$, as we needed to show.
\end{proof}

\begin{remark}
\label{remsemifact} 
For all $i \in I,$ denote by 
$
\pi_i\colon P_i \rightarrow P $  and $ \alpha_i\colon \cC \rightarrow \cC_i
$ the  universal morphisms. 
Then it follows from the proof of  Proposition  \ref{limit of psod} that if $w$ is in $P,$ then $A \in \cC$ lies 
in $\cC_w \subseteq \cC$ if and only if, for all $i \in I$ the image
$\alpha_i(A)$ lies in the subcategory $\langle \cC_{i, x} \, , \,  x \in \pi_i^{-1}(w) \rangle \subseteq \cC_i.$ 
\end{remark}

\subsection{Gluing psod-s and conservative descent}
\label{gpacd}
A formalism for gluing semi-orthogonal decompositions along faithfully flat covers was proposed in  \cite[Theorem B]{bergh2017conservative}. The authors call their theory  \emph{conservative descent}. The proof given  
in  \cite{bergh2017conservative} depends on rather subtle arguments. The key difference with our approach is that in that paper, the authors work with the classical theory of triangulated categories, for which there is no well-behaved notion of limits and colimits. Using the full power of the $\infty$-category of dg categories we can sidestep these  difficulties, and give a simple and conceptual proof  of conservative descent. From our perspective, conservative descent becomes a special case of the general structure result for limits of categories equipped with a psod given by Theorem \ref{limitpsod2}. 

More precisely, we will show that our Proposition \ref{limit of psod} and Theorem \ref{limitpsod2} imply Theorem B from \cite{bergh2017conservative}. 
We start by briefly recalling the setting of \cite{bergh2017conservative}, referring the reader to  the original reference for full details. Let $S$ be an algebraic stack. Let $X$ and $Z_1, \ldots, Z_m$ be algebraic stacks over $S$, and let  $S' \to S$ a faithfully flat map. If $T$ is an algebraic stack, we denote by $ \mathrm{D}_{\mathrm{qc}}(T)$ the classical derived category of quasi-coherent sheaves over $T$. Note that the category 
 $ \mathrm{D}_{\mathrm{qc}}(T)$ is the homotopy category of 
 $\Qcoh(T)$.

 The set-up of 
 Theorem B requires to consider
 \begin{itemize} 
 \item functors $\Phi_i\colon \mathrm{D}_{\mathrm{qc}}(Z_i) \to  \mathrm{D}_{\mathrm{qc}}(X)$ which are of \emph{Fourier-Mukai} type  (in the sense of Definition 3.3 of \cite{bergh2017conservative}),
 \item and their \emph{base change} along $S' \to S$. If we set $X' = X \times_S S'$, and $Z'_i = Z_i \times_S S' $, then the base change of $\Phi_i$ is a functor $\Phi_i'\colon\mathrm{D}_{\mathrm{qc}}(Z'_i) \to  \mathrm{D}_{\mathrm{qc}}(X')$. 
 \end{itemize} 
Then Theorem B breaks down as the following two statements: 
\begin{enumerate} 
\item If the functors $\Phi_i'$ are fully-faithful, then the functors $\Phi_i$ are also fully-faithful. Under this assumption, if the subcategories  $\mathrm{Im}(\Phi_i')$ are semi-orthogonal in $\mathrm{D}_{\mathrm{qc}}(X')$,  
$$
\text{i.e.} \quad \mathrm{Im}(\Phi_j') \subset \mathrm{Im}(\Phi_{j'}')^\bot \quad \text{if} \quad j < j', \quad
$$
then   $\mathrm{Im}(\Phi_i)$ are  semi-orthogonal in $\mathrm{D}_{\mathrm{qc}}(X)$. 
\item Moreover if the subcategories 
$\mathrm{Im}(\Phi_i')$ generate $\mathrm{D}_{\mathrm{qc}}(X')$, then the subcategories 
$\mathrm{Im}(\Phi_i)$ generate $\mathrm{D}_{\mathrm{qc}}(X)$. 
 \end{enumerate}
 Let us sketch how to recover these results from Proposition \ref{limit of psod} and Theorem  \ref{limitpsod2}. We will break this explanation down into several steps. For the sake of   clarity we will gloss over some technical details, which will be left to the reader. 
  \begin{itemize}[leftmargin=*]
 \item We set $\cC:= \Qcoh(X)$. For all $k \in \mathbb{N}$, we denote by $\cC_k$ the dg-category of quasi-coherent sheaves over the  $k$-th iterated fiber product of $X$ and $S'$ over $S$
 $$
 \cC_k := \Qcoh(X \times_{S} S' \times_{S}  \ldots  \times_{S} S').
 $$
\item For all $k \in \mathbb{N}$  we denote by $\cC_{k, i}$ the dg-category of quasi-coherent sheaves over the  $k$-th iterated fiber product of $Z_i$ and $S'$ over $S$
 $$
 \cC_{k, i} := \Qcoh(Z_{i} \times_{S} S' \times_{S}  \ldots  \times_{S} S').  
 $$
 \item Since $\Phi_i$ is of Fourier--Mukai type it lifts to a functor between the dg-enhancements of the derived categories of quasi-coherent sheaves. We keep denoting these functors $\Phi_i$ 
  $$
 \Phi_i\colon \Qcoh(Z_i) \to \Qcoh(X).
 $$
By base change for all $k \in \mathbb{N}$  we get functors 
 $\, \Phi_{k, i}\colon \cC_{k, i} \to \cC_k.$  
 \item As in Theorem B from \cite{bergh2017conservative} we assume that for all $i \in \{1, \ldots, m\}$ $\Phi_{1,i}=\Phi'_i$ is fully-faithful, and that 
 $\cC_{1,i}\simeq \mathrm{Im}(\Phi'_i)$ are semi-orthogonal in $\cC_1=\Qcoh(X')$. Equivalently, this    can be formulated  by saying that the subcategories $\cC_{1,i}$ form a pre-psod of 
  $\cC_1$ of type 
 $P$, where $P$ is the set $\{1, \ldots, m\}$ with the usual ordering. One can show that this implies that, for all $k \in \mathbb{N}$, the subcategories $ \cC_{k, i}$ form a pre-psod on 
 $\cC_k$ of type 
 $P$.  
  \item By faithfully flat descent,  $\Qcoh(X)$ can be obtained   as the 
 limit 
\begin{equation}
\label{cechnerve}
\Qcoh(X) \longrightarrow \big [\Qcoh(X' )    
\substack{\longrightarrow\\[-0.6em]   \\[-1em] \longrightarrow } \Qcoh(X' \times_S S')  \substack{\longrightarrow\\[-1em] \longrightarrow \\[-1em] \longrightarrow } \   \Qcoh(X' \times_S S' \times_S S' ) \ldots \big ]. 
\end{equation}
In formulas we will write 
$ \, \cC = \varprojlim_{k \in \mathbb{N}} \cC_k \,$. For every $k \in \mathbb{N}$ we denote by 
$$a_{k, 1}, \ldots, a_{k, k+1}\colon \cC_k \rightarrow \cC_{k+1}$$
the $k+1$ structure maps from diagram (\ref{cechnerve}).

Let $k\in \mathbb{N}$, and let $j \in \{ 1, \ldots, k+1\}$. 
It is easy to see that the identity $\, \mathrm{id}\colon P  \to P$ is an order-reflecting map and  equips $a_{k,j}\colon (\cC_k, P) \to (\cC_{k+1},P)$ with a structure of ordered functor.    
\item Proposition \ref{limit of psod} then implies that $ \, \cC = \varprojlim_{k \in \mathbb{N}} \cC_k \,$ carries a pre-psod of type    $$\varinjlim_{k \in \mathbb{N}}P=P$$    with semi-orthogonal factors given by 
$\varprojlim_k\cC_{k, i} \simeq \Qcoh(Z_i)$. This recovers statement $(1)$ of Theorem B. Now assume that $\cC_{1,i}$ is actually a psod of $\cC_1=\Qcoh(X')$. This implies that for all $k\in \bN$, $\cC_{k, i}$ is a psod of type $P$ of $\cC_k$.
Then Theorem \ref{limitpsod2} implies statement $(2)$ of Theorem B.\footnote{Proposition \ref{limit of psod} and Theorem \ref{limitpsod2} were formulated for \emph{small} dg-categories, while here we are applying them to the \emph{large} category $\Qcoh(X)$: however our results also hold, without variations, in the setting of large categories.}
 \end{itemize}
 
\section{Semi-orthogonal decompositions  of root stacks}
\label{sec:sodros}
In this section we explain our main application. Let $X$ be an algebraic stack and let $D$ be a normal crossings divisor in $X$. 
The root stack $\radice[r]{(X,D)}$ of a pair $(X, D)$, where $D \subset X$ is a normal crossings divisor, has long been an important object in algebraic geometry. We refer the reader to the Introduction for more information on previous work in this area. We  will construct semi-orthogonal decompositions on 
$\Perf(\radice[r]{(X,D)})$. This generalizes earlier results by other authors. In \cite{ishii2011special} and \cite{bergh2016geometricity} the authors construct sod-s on $\Perf(\radice[r]{(X,D)})$, under the assumption that $D$ is \emph{simple} normal crossings.   We drop the assumption of simplicity and work with general normal crossings divisors.

\begin{remark}
We want to stress a subtle point about this assumption: from \cite[Definition 3.5]{bergh2016geometricity}, it might seem that in that paper they \emph{do} consider divisors that are merely normal crossings.

The point is that the root construction that they use in the non-simple case is not the ``correct'' one from the point of view of logarithmic geometry. For instance, if $D\subset \bP^2$ is an irreducible nodal cubic, their $r$-th root construction would add a stabilizer $\mu_r$ along all the points of $D$, including the node. This does not take into account that locally around the node there are two distinct branches of $D$. In this case, the ``correct'' construction of $\radice[r]{(\bP^2,D)}$ (and more generally when $D$ is normal crossings but not simple) from our point of view is the one introduced in \cite{borne-vistoli}, that adds a stabilizer $\mu_r$ along all the points of $D$ different from the node, and a stabilizer $\mu_r^2$ at the node. A way to see that this is indeed the right notion is that, following the definition of \cite{borne-vistoli}, the stack $\radice[r]{(\bP^2,D)}$ is smooth, just as the root stacks of smooth schemes along simple normal crossings divisors are; whereas with the definition of root stacks  given in \cite[Definition 3.5]{bergh2016geometricity}, in the non-simple normal crossings case one obtains singular stacks. 
\end{remark}

 The shape of the sod-s that we construct in the general case have  interesting differences from  the ones given in  \cite{ishii2011special} and \cite{bergh2016geometricity}. Indeed whereas in the simple normal crossings case the factors of the sod-s are equivalent to perfect complexes on the strata, in the general case we need to work with the \emph{normalizations} of the strata.

Our construction of psod-s for $\Perf(\radice[r]{(X,D)})$  relies in essential way on Theorem \ref{limitpsod2}. The other main ingredient are the psod-s  obtained in  \cite{bergh2016geometricity} in the simple normal crossing case,  and which will be reviewed in a slightly different formulation in Section \ref{secsimple} below. 

\subsection{Root stacks of normal crossing divisors}
\label{secsimple}
We start by introducing some notations. Let $X$ be an algebraic stack and let $D \subset X$ be a (non-necessarily simple) normal crossing divisor. Note that $X$ carries a natural stratification given by locally closed substacks which can be obtained, locally, as intersections of the branches of $D$. Equivalently this  stratification can be defined as follows: let $\widetilde{D}$ be the normalization of $D$ and consider the locally closed substacks $S$ which are maximal with respect to the following two properties: $S$ is connected;  the map $\widetilde{D} \times_X S \to S$ is \'etale. 

Let $\cS(D)$ denote the set of \emph{strata closures}, that we will henceforth simply call ``strata''.\footnote{In the usual definition, strata are locally closed, and their boundary is a disjoint union of smaller dimensional strata. We look instead at closed strata, which are given by the closures of the locally closed ones.} For every $k \in \mathbb{N}$ we set 
$$\cS(D)_k := \{S \in \cS(D) \mid \mathrm{codim}(S)=k\}, \quad \text{and} \quad S(D)_k := \bigcup_{S \in \cS(D)_k} S.
$$
We say that $S(D)_k$ is the \emph{$k$-codimensional skeleton} of the stratification.  For every stratum $S \in \cS(D)$ we denote by $\widetilde{S}$ its 
normalization. 
We denote by $\widetilde{S(D)}_k$ the \emph{normalization} of $S(D)_k$:  $\widetilde{S(D)}_1$ is the disjoint union of the \emph{normalizations} of the irreducible components of 
$D$; more generally, $\widetilde{S(D)}_k$ is the disjoint union of the \emph{normalizations} of the strata in $\cS(D)_k$. In formulas we can write
$$
\widetilde{S(D)}_k = \coprod_{S \in \cS(D)_k} \widetilde{S}.
$$
Let $N_D$ be the maximal codimension of  strata of $(X, D)$. We denote by 
$(\underline{N_D}, \leq)$ the  ordered set
\[
\underline{N_D} := \{N_D, N_D - 1, \ldots, 0 \} \quad \text{ordered by} \quad 
N_D < N_D - 1 < \ldots < 0.
\]
We set $\underline{N_D}^*=\underline{N_D}-\{0\}$. 
For $r\in \mathbb{N}$, we denote by $\mathbb{Z}_r$ the group of residue classes modulo $r$ and set $\mathbb{Z}_r^*:=\mathbb{Z}_r - \{0\}$. As in \cite{SST2},  it is useful to 
identify  
$\mathbb{Z}_r$ and $\mathbb{Z}_r^*$ with subsets of $\mathbb{Q}/\mathbb{Z}  = \mathbb{Q} \cap (-1,0]$ 
\begin{equation}
\label{eqpreorder1}
\mathbb{Z}_r \cong \left\{-\frac{r-1}{r}, \hdots, -\frac{1}{r}, 0\right\} \subset \mathbb{Q} \cap (-1,0], \quad  
\mathbb{Z}_r^* \cong \left\{-\frac{r-1}{r}, \hdots, -\frac{1}{r}\right\} \subset \mathbb{Q} \cap (-1,0].
\end{equation}
We equip $\mathbb{Z}_r$  and $\mathbb{Z}_r^*$ with the order $\leq$ given by the restriction of the order on $\mathbb{Q}$  
$$-\frac{r-1}{r} < -\frac{r-2}{r} < \hdots < -\frac{1}{r} < 0.$$
For all $k \in \underline{N_D}^*$ we set 
$ \,  \mathbb{Z}_{k, r} := \bigoplus_{i=1}^{k} \mathbb{Z}_{r} \, $ and $ \, 
\mathbb{Z}_{k, r}^* := \bigoplus_{i=1}^{k} \Big ( \mathbb{Z}_{r} - \{0\} \Big )\, $. Note that $\mathbb{Z}_{0, r}^*=\{0\}$. We equip  $\mathbb{Z}_{k, r}$  and $ 
\mathbb{Z}_{k, r}^*$ with the product preorder. The group 
$\mathbb{Z}_{k, r}$ is canonically isomorphic to the group of characters of the group of roots of unity $\mu_{k, r}:=\bigoplus_{i=1}^{k} \mu_{r}$.

We denote by $(\mathbb{Z}_{D,r}, \leq)$ the   set 
$ 
\coprod_{k=0}^{N_D} \mathbb{Z}_{k, r}^*
$ equipped with the following preorder: let 
 $\xi, \xi'$ be in $\mathbb{Z}_{D,r}$, with $\xi \in \mathbb{Z}_{k, r}^*$ and 
$\xi' \in \mathbb{Z}_{k', r}^*$, then  
\[
\xi < \xi' \text{ if  }  k <_{\underline{N_D}}k' \text{  or   }  k = k'     \text{  and   }  \xi <_{\mathbb{Z}_{k, r}^*} \xi'.
\]
Let $\radice[r]{(X,D)}$ be the $r$-th root stack of $(X,D)$. 
Note that $\radice[r]{(X,D)}$ also carries a natural normal crossing divisor, which we denote $D_r \subset \radice[r]{(X,D)}$, obtained as reduction of the preimage of $D$. All previous notations and definitions therefore also apply to the log pair $(\radice[r]{(X,D)}, D_r)$.

\begin{proposition}
\label{prop: sodvnc}
\begin{enumerate}[leftmargin=*] 
\item The category $\Perf(\radice[r]{(X,D)})$ has a  psod
 of type $(\underline{N_D}, \leq)$,  
  $\, 
\Perf(\radice[r]{(X,D)}) = \langle \cA_k, k \in \underline{N_D} \rangle 
\,$ 
where:
\begin{itemize}[leftmargin=*]
\item $\cA_{ 0}\simeq \Perf(X)$.
\item  For all $k \in  \underline{N_D}^*$, the subcategory $\cA_k$  has a psod of type $(\mathbb{Z}_{k, r}^*, \leq)$ 
$$
\cA_k = \langle  \cA_{ \chi}^k, \chi \in (\mathbb{Z}_{k,r}^*, \leq)  \rangle
$$ 
and, for all $\chi \in \mathbb{Z}_{k, r}^*$, there is an equivalence $\cA_{ \chi}^k  \simeq \Perf(\widetilde{S(D)_k})$.
\end{itemize}
\item The  category $\Perf(\radice[r]{(X,D)})$ has a psod of type $(\mathbb{Z}_{D,r}, \leq)$ 
$$\Perf(\radice[  r]{(X,D)})=  \langle  \cA_\chi^k, \chi \in  (\mathbb{Z}_{D,  r}, \leq)   \rangle,$$
where $\cA_{ 0}\simeq \Perf(X)$ and for all $\chi \in \mathbb{Z}_{k, r}^*$ there is an equivalence $\cA_{ \chi}^k \simeq \Perf(\widetilde{S(D)}_k).$
\end{enumerate}
\end{proposition}

Before giving a proof of Proposition \ref{prop: sodvnc} we will make a few preliminary considerations.

If $D \subset X$ is a \emph{simple normal crossing} divisor, 
Proposition \ref{prop: sodvnc}  is a reformulation of  Theorem 4.9 from 
\cite{bergh2016geometricity}. In order to translate back to the statement of \cite[Theorem 4.9]{bergh2016geometricity}, it is sufficient to note that $\widetilde{S_k}$ decomposes as the disjoint union of strata of codimension $k$.  Indeed, in the simple normal crossing setting, all strata are already normal. This yields an equivalence 
$$
\cA_{ \chi}^k \simeq \bigoplus_{S \in \cS_k} \Perf(S)
$$
which recovers the semi-orthogonal factors given by \cite[Theorem 4.9]{bergh2016geometricity}.

It will be useful to explain how the semi-orthogonal factors 
$\cA_\chi^k$ are constructed in the simple normal crossing case, and some of their basic properties. We refer to the treatment contained in Sections 3.2.1 and 3.2.2  of  \cite{SST2}, and limit ourselves to statements without proof. For all $k \in \mathbb{N}$ there is a canonical $\mu_{k, r}$-gerbe $G_{k,r}(D) \to \widetilde{S(D)}_k$, that fits in a diagram 
$$
G_{k,r}(D) \stackrel{q} \longleftarrow \widetilde{S(D_r)}_k \stackrel{\iota}  \longrightarrow \Perf(\radice[r]{(X,D)}) $$
where $\widetilde{S(D_r)}_k$ is the normalization of the codimension $k$ skeleton of the stratification of $(\radice[r]{(X,D)}, D_r)$. 
There is a natural splitting (as in Lemma 3.9 of \cite{SST2})
$$
\Perf(G_{k,r}(D)) \simeq \bigoplus_{\chi \in \mathbb{Z}_{k, r}} \Perf(G_{k,r}(D))_\chi 
 \simeq \bigoplus_{\chi \in \mathbb{Z}_{k, r}} \Perf(\widetilde{S(D)}_k).
$$
For every $\chi \in \mathbb{Z}_{k, r}$, the subcategory 
$\cA^ k_\chi$ is defined as the image of  the composite
\begin{equation}
\label{semifactorseq}
  \cA^ k_\chi  = \Perf(\widetilde{S(D)}_k)  \simeq \Perf(G_{k, r}(D))_\chi  \xrightarrow{(\star)}\Perf(G_{k, r}(D)) \xrightarrow{\iota_{*}q^*} \Perf(\radice[r]{(X,D)})  
\end{equation}
where  arrow $(\star)$ is the inclusion of the  
$\chi$-th factor.

 \begin{proof}[Proof of Proposition \ref{prop: sodvnc}] 
We will use the fact that, by
\cite[Theorem 4.9]{bergh2016geometricity},  the statement holds when $D$ is  simple normal crossing.

Let $D \subset X$ be a general normal crossing divisor. 
Consider  an   \'etale covering $U \to X$ such that the pull-back to $U$ of the log structure of $(X, D)$ is induced by a  
 simple normal crossings divisor 
$D|_U$.   
The construction of root stacks is compatible with base change, so that there are natural isomorphisms $\radice[r]{(U,D|_{U})} \simeq \radice[r]{(X,D)}\times_X U$.  Further,   $\radice[r]{(U,D|_{U})} \to \radice[r]{(X,D)}$ is an \'etale covering. For all $l \in \mathbb{N}$ we denote the $l$-fold iterated fiber product of $U$ over $X$ 
$ \, \, 
U_{l}:= U  \times_X  \ldots 
\times_X U  \, \,.$  For all $l \in \mathbb{N}$ the pull-back  of the log structure 
on $(X, D)$  to $U_l$ is also induced 
by a simple normal crossing divisor, which we denote $D_l$.

Consider the semi-simplicial stack given by the nerve of the \'etale cover $ \radice[r]{(U,D|_U)}  \to \radice[r]{(X,D)}$ 
\[
\ldots \substack{\longrightarrow\\[-1em] \longrightarrow \\[-1em] \longrightarrow } \, \, \radice[r]{(U_2,D_2)}   \rightrightarrows \radice[r]{(U,D|_{U})}  \longrightarrow \radice[r]{(X,D)}. 
\]
Note that for all $l$, the structure morphisms $p_1, \ldots, p_{l+1}$ 
 $$
p_1, \ldots, p_{l+1}\colon \radice[r]{(U_{l + 1},D_{l + 1})} \longrightarrow 
\radice[r]{(U_{l },D_{l})} 
 $$
 map strata of codimension $k$ to strata of codimension $k$. By faithfully flat descent, we can realize $\Perf(\radice[r]{(X,D)})$ as the totalization of the induced semi-cosimplicial diagram of dg-categories, where the structure maps are given by pull-back functors: 
\begin{equation}
\label{limitofcat}
\Perf(\radice[r]{(X,D)}) \simeq \varprojlim_{l \in \mathbb{N}} \Perf \left(\radice[r]{(U_{  l},D|_{U_{ l}})}  \right). 
\end{equation}
We are going to prove the proposition by applying Theorem \ref{limitpsod2} to the limit of dg-categories (\ref{limitofcat}), more precisely:
\begin{itemize} 
\item Since $D_l$ is simple normal crossing, for every $l$ we can equip $\Perf(\radice[r]{(U_{ l},D|_{U_{ l}})})$ with the psod of type $\mathbb{Z}_{D_l,  r}$ given by Proposition \ref{prop: sodvnc}, where the semi-orthogonal factors are defined as in (\ref{semifactorseq}). Also, since $\mathbb{Z}_{D_l,  r} = \mathbb{Z}_{D,  r}$, we can write 
$$
\Perf(\radice[r]{(U_{l},D_l)} ) = \langle   \cA_\chi^{l, k}, \chi \in  (\mathbb{Z}_{D,  r}, \leq)   \rangle.
$$
\item We equip the functors appearing in   limit  (\ref{limitofcat}) with the ordered structure given by the identity map: that is, for every $j \in \{1, \ldots, l+1\}$ we have 
$$
p_j^*\colon \Perf(\radice[r]{(U_{l },D_{l})}) \longrightarrow 
 \Perf(\radice[r]{(U_{l + 1},D_{l + 1})} )
$$
and we set $\phi_{p_j^*}=\mathrm{id}\colon \mathbb{Z}_{D,  r} \to \mathbb{Z}_{D,  r}.$
\end{itemize}
Note that $\varinjlim{\mathbb{Z}_{D_l,r}}=\mathbb{Z}_{D,  r}$ is a finite preorder. Thus, to apply Theorem \ref{limitpsod2}, we only need to check  
 that the identity map $\mathrm{id}\colon \mathbb{Z}_{D,  r} \to \mathbb{Z}_{D,  r}$ does indeed induce a well-defined ordered structure on the functors appearing in (\ref{limitofcat}). Namely, we have to prove the following two properties: 
\begin{enumerate}[leftmargin=*]
\item[$(a)$]  We need to show that for all 
$j \in \{1, \ldots, l+1\}$, for all $\chi \in \mathbb{Z}_{D,  r}$ there is a strictly commutative diagram 
\begin{equation}
\begin{gathered}
\label{eqorderedid}
\xymatrix{\cA_\chi^{l+1, k} \ar[r] & 
 \Perf(\radice[r]{(U_{l + 1},D_{l + 1})} )\\
\cA_\chi^{l, k} \ar[r] \ar[u]^-{p_j^*} \ocell{ur} & \Perf(\radice[r]{(U_{l },D_{l})}).  \ar[u]_-{p_j^*} }
\end{gathered}
\end{equation}

\item[$(b)$] Additionally we need to show that the canonical 2-cells obtained  via adjunction from (strictly) commutative diagram (\ref{eqorderedid}) are also invertible, giving rise to commutative diagrams 
\begin{equation}
\begin{gathered}
\label{leftrightadjoints}
\xymatrix{\cA_\chi^{l+1, k} & 
 \Perf(\radice[r]{(U_{l + 1},D_{l + 1})} ) \ar[l]_-{l_\chi} && \cA_\chi^{l+1, k}   & 
 \Perf(\radice[r]{(U_{l + 1},D_{l + 1})} ) \ar[l]_-{r_\chi} \\
\cA_\chi^{l, k}  \ar[u] & \Perf(\radice[r]{(U_{l },D_{l})})  \ar[u]_-{p_j^*} \ar[l]_-{l_\chi} && \cA_\chi^{l, k}  \ar[u] & \Perf(\radice[r]{(U_{l },D_{l})}).  \ar[u]_-{p_j^*} \ar[l]_-{r_\chi}}
\end{gathered}
\end{equation}
\end{enumerate}
Let us start with property $(a)$. The horizontal arrows in (\ref{eqorderedid}) are inclusions, thus we only need to check that $p_j^*$ maps $\cA_\chi^{l, k}$ to the subcategory $\cA_\chi^{l+1, k}$. Note that 
the map  
$$p_j\colon \radice[r]{(U_{l + 1},D_{l + 1})} \to 
\radice[r]{(U_{l },D_{l})}$$ maps strata to strata. Thus we get a commutative diagram 
$$
\xymatrix{
G_{k,r}(D_{l+1}) \ar[d]_-{p_j} & \ar[l]_-{q}  \widetilde{S(D_{l+1, r})}_k \ar[r]^-{\iota}  \ar[d]^-{p_j} & \Perf(\radice[r]{(U_{l+1},D_{l+1})}) \ar[d]^-{p_j} \\
G_{k,r}(D_{l}) & \ar[l]_-{q}  \widetilde{S(D_{l, r})}_k \ar[r]^-{\iota}  & \Perf(\radice[r]{(U_{l},D_{l})}) }
$$
where all vertical arrows are \' etale (and in particular flat and proper), and both  the left and the right square are fiber products. 
Property $(a)$  follows because there is a natural equivalence  
$$
\iota_*q^*p_j^*  \simeq p_j^*\iota_*q^*\colon G_{k,r}(D_{l}) \longrightarrow \Perf(\radice[r]{(U_{l+1},D_{l+1})})
$$
given by the   composite   
$$
\iota_*q^*p_j^* \simeq \iota_*p_j^*q^* \stackrel{(\star)}\simeq p_j^*\iota_*q^*  
$$
 where equivalence $(\star)$ is given by flat base change.

 Let us consider property $(b)$ next. We will show that the Beck--Chevalley property holds with respect to right adjoints: that is, that the right square in (\ref{leftrightadjoints}) commutes. The case of left adjoints is similar. Note that 
 $ 
 r_\chi\colon
 \Perf(\radice[r]{(U_{l + 1},D_{l + 1})} )  \to \cA_\chi^{l+1, k}  
$ 
is given by the composite
$$
 \Perf(\radice[r]{(U_{l+1},D_{l+1})}) \xrightarrow{q_*\iota^ !}  
G_{k,r}(D_{l+1}) \xrightarrow{\mathrm{pr}_\chi} \cA_\chi^{l+1, k}, 
$$
where $\mathrm{pr}_\chi$ is the projection onto the $\chi$-th factor, and similarly for 
$  
r_\chi\colon \Perf(\radice[r]{(U_{l },D_{l})})  \to \cA_\chi^{l, k} 
$. Thus, in order to show that the right square in (\ref{leftrightadjoints}) commutes, it is enough to prove that the canonical natural transformation 
$$
q_*\iota^!p_j^* \Rightarrow p_j^*q_*\iota^!\colon \Perf(\radice[r]{(U_{l},D_{l})}) \longrightarrow  G_{k,r}(D_{l+1}) 
$$
is invertible. This can be broken down as a composite of the base change 
natural transformations 
$$
q_*\iota^!p_j^* \stackrel{(i)}\Rightarrow q_*p_j^*\iota^! \stackrel{(ii)}\Rightarrow p_j^*q_*\iota^!
$$
which are both equivalences: $(i)$ is an equivalence by proper base change 
($\iota^!p_j^* \simeq p_j^*\iota^!$), and $(ii)$ by flat base change ($q_*p_j^* \simeq p_j^*q_*$).

Since property $(a)$ and $(b)$ are satisfied we can apply Theorem \ref{limitpsod2} to our setting. Thus 
$\Perf(\radice[r]{(X,D)})$ carries a psod of type  $(\mathbb{Z}_{D,r}, \leq)$ 
$$\Perf(\radice[  r]{(X,D)})=  \langle  \cA_\chi^k, \chi \in  (\mathbb{Z}_{D,  r}, \leq)   \rangle$$
where for all $\chi \in \mathbb{Z}_{D,  r}$ we have that 
$$
\cA_\chi^k = \varprojlim_{l \in \mathbb{N}} \cA_\chi^{l,k} \simeq \varprojlim_{l \in \mathbb{N}} 
\Perf(\widetilde{S(D_l)}_k).
$$
Note that the stacks $\widetilde{S(D_l)}_k$, together with the structure maps between them, are the nerve of the \'etale cover 
$\widetilde{S(D_U)}_k \to \widetilde{S(D)}_k$. Thus, by faithfully flat descent, we have an equivalence
$$
\cA_\chi^k  \simeq \varprojlim_{l \in \mathbb{N}} 
\Perf(\widetilde{S(D_l)}_k) \simeq \Perf(\widetilde{S(D)}_k).  
$$
This concludes the proof of part $(2)$ of Proposition \ref{prop: sodvnc}, which immediately implies part $(1)$. 
 \end{proof}

It will be useful to reformulate Proposition \ref{prop: sodvnc} in a way that is closer 
to the analogous statements in the simple normal crossing setting given in \cite[Theorem 4.9]{bergh2016geometricity} and \cite[Proposition 3.12]{SST2}. We do this in Corollary \ref{sodncver2} below. 
This amounts  to breaking down the factors 
$\cA_{ \chi}^k  \simeq \Perf(\widetilde{S(D)}_k)$ into direct sums: since $\widetilde{S(D)}_k$ is the disjoint union of the normalizations of the $k$-codimensional strata, we have a decomposition 
$$
\cA_{ \chi}^k  \simeq \Perf(\widetilde{S(D)}_k) \simeq \bigoplus_{S \in \cS(D)_k} \Perf(\widetilde{S})
$$
Before stating this result we need to introduce some notations. We equip the set of strata $\cS(D)$ with the   coarsest preorder  
 satisfying the following two properties: let 
 $S$ and $S'$ be in $\cS(D)$
\begin{enumerate}
\item if   $S \subseteq S'$ then $S \leq S'$, and
\item if $\mathrm{dim}(S)=\mathrm{dim}(S')$ then $S \leq S'$ and $S'\leq S$.
\end{enumerate} 
We denote $\cS(D)^* := \cS(D) - \{X\}$.  For every $S \in \cS(D)$, if $\mathrm{cod}(S)$ is the codimension of $S$ we set  
$|S|:= \mathrm{cod}(S)$.
We denote by $(\mathbb{Z}_{\cS(D),r}, \leq)$ the   set 
$ 
\coprod_{S \in \cS(D)} \mathbb{Z}_{|S|, r}^*
$ equipped with the following preorder: let 
 $\xi, \xi'$ be in $\mathbb{Z}_{\cS{D},r}$, with $\xi \in \mathbb{Z}_{|S|, r}^*$ and 
$\xi \in \mathbb{Z}_{|S'|, r}^*$, then  $\xi < \xi'$  
\begin{itemize}
\item if $|S| > |S'|$,
\item or if $|S| = |S'|$ and $S \neq S'$,
\item or if $S=S'$ and $\xi <_{\mathbb{Z}_{|S|, r}^*} \xi'$.
\end{itemize}
Corollary \ref{sodncver2} generalizes \cite[Theorem 4.9]{bergh2016geometricity} and \cite[Proposition 3.12]{SST2} to pairs $(X, D)$ where $D$ is normal crossing but necessarily simple.  

\begin{corollary}
\label{sodncver2}
 \mbox{ }
\begin{enumerate}[leftmargin=*] 
\item The category $\Perf(\radice[r]{(X,D)})$ has a  psod
 of type $(\cS(D), \leq)$,  
  $\, 
\Perf(\radice[r]{(X,D)}) = \langle \cA_S, S \in \cS(D) \rangle 
\,$ 
where: 
\begin{itemize}[leftmargin=*]
\item $\cA_{X}\simeq \Perf(X)$.
\item  For all $S \in  \cS(D)^*$, the subcategory $\cA_S$  has a psod of type $(\mathbb{Z}_{|S|, r}^*, \leq)$ 
$$
\cA_S = \langle  \cA_{ \chi}^S, \chi \in (\mathbb{Z}_{|S|,r}^*, \leq)  \rangle
$$ 
and, for all $\chi \in \mathbb{Z}_{|S|, r}^*$, there is an equivalence $\cA_{ \chi}^S \simeq \Perf(\widetilde{S})$. 
\end{itemize}
\item The  category $\Perf(\radice[r]{(X,D)})$ has a psod of type $(\mathbb{Z}_{\cS(D),r}, \leq)$ 
$$\Perf(\radice[  r]{(X,D)})=  \langle  \cA_\chi^S, \chi \in  (\mathbb{Z}_{\cS(D),r}, \leq)  \rangle,$$
where $\cA^{X} \simeq \Perf(X)$ and for all $\chi \in \mathbb{Z}_{|S|, r}^*$ there is an equivalence $\cA_{ \chi}^S \simeq \Perf(\widetilde{S}).$
\end{enumerate}
\end{corollary} 

\begin{example}
It might be useful to describe  the sod  
given by Corollary \ref{sodncver2} in a concrete example. Let $X=\mathbb{A}^2$ and let $D$ be an irreducible nodal cubic curve with node at the origin $o \in \mathbb{A}^2$. The stratification of $\mathbb{A}^2$ induced by $D$ has three strata: $o$, $D$ and 
$\mathbb{A}^2$. The normalization of $\widetilde{D}$ of $D$ is isomorphic to $\mathbb{A}^1$. The preorder $(\mathbb{Z}_{\cS(D),2},  \leq)$  is 
given by
$ \, 
o < D < \mathbb{A}^2
$.  
Then Corollary \ref{sodncver2} yields a sod of the form 
$$
\Perf(\sqrt[2]{(\mathbb{A}^2, D)}) = \langle \, \Perf(o), 
\Perf(\widetilde{D}), \Perf(\mathbb{A}^2) \, \rangle. 
$$
\end{example}

 \section{Applications to logarithmic geometry}
 In this last section we explain two consequences of our results in the context of logarithmic geometry and the theory of parabolic sheaves. In section \ref{secirsgenncd} we construct an infinite sod on the category of perfect complexes over the \emph{infinite root stack} $\radice[\infty]{(X,D)}$ of a pair $(X, D)$ where $D$ is a general normal crossing divisor. Equivalently, this can be expressed by saying that we construct an infinite sod on the derived category of parabolic sheaves with rational weights on $(X, D)$. This improves earlier results that we obtained in  \cite{SST2}. Then in section \ref{kfkoncd} we show that this implies   a generalization to the general normal crossings case of an important structure result in Kummer flat K-theory originally due to Hagihara and Nizio\l.

\subsection{Infinite root stacks of  normal crossing divisors}
\label{secirsgenncd}
Before proceeding it is useful to recall our results  from \cite{SST2}, and explain in which  way our current results improve on them. 
\begin{enumerate}
\item In  \cite{SST2} we construct  psod-s on $\Perf(\radice[\infty]{(X,D)})$, under the assumption that $D$ is simple normal crossings. 
\item In \cite{SST2} we also constructed psod-s on $\Perf(\radice[\infty]{(X,D)})$ in the general normal crossing case. This however required, first of all, to work over a field of characteristic zero, and secondly involved using 
 a highly non-trivial result on the  invariance of $\Perf(\radice[\infty]{(X,D)})$ under log blow-ups which was established in \cite{scherotzke2016logarithmic}. In particular the sod obtained in this way depended on a choice of a simple normal crossing model $(\widetilde{X}, \widetilde{D})$ obtained from $(X,D)$ by successive blow-ups along the strata of $D$. 
\end{enumerate}

 \begin{remark}
It is also important to note that the psod-s in the general normal crossing case constructed in \cite{SST2} only exist for the infinite root stack, and not for finite root stacks. The argument in \cite{SST2} 
 requires   to switch to a simple normal crossing birational model $(\widetilde{X}, \widetilde{D})$  of $(X, D)$. The key point is that the categories of perfect complexes of  $\radice[\infty]{(\widetilde{X}, \widetilde{D})}$  and $\radice[\infty]{(X,D)}$ are equivalent by the main theorem of \cite{scherotzke2016logarithmic} but this is far from true for finite root stacks. Thus the psod-s on  finite root stacks constructed in  Proposition \ref{prop: sodvnc} are entirely new, even over a field of characteristic 0. 
\end{remark}

Let $(X,D)$ be a  pair given by an algebraic stack equipped with a normal crossing divisor.  Our main goal in this section is to use the results of Section \ref{sec:sodros} to construct psod-s (patterned after Corollary 
\ref{sodncver2}) on $\Perf(\radice[\infty]{(X,D)})$ in the general normal crossing case, which are independent of the ground ring and of the choice of a 
desingularization $(\widetilde{X}, \widetilde{D})$.  
 This is given by Theorem \ref{mainsgncdivinf} below. 
The proof strategy is  the same as the one that was used to prove Theorem 3.16 of \cite{SST2}, whose statement exactly parallels   Theorem \ref{mainsgncdivinf}:    
the only difference is that Theorem 3.16 of \cite{SST2}
assumes that $D$ is simple normal crossings. 
For this reason we will limit ourselves to state our results, referring the reader to \cite{SST2} for the proof. 

Formulating Theorem \ref{mainsgncdivinf} requires introducing some notations. Let $S$ be a stratum in $\cS(D)$. We denote 
 $$
(\mathbb{Q}/\mathbb{Z})_{|S|} := \bigoplus_{i=1}^{|S|} \mathbb{Q}/\mathbb{Z}, \quad 
(\mathbb{Q}/\mathbb{Z})_{|S|}^*:=  \bigoplus_{i=1}^{|S|} \Big ( \mathbb{Q}/\mathbb{Z}- \{0\} \Big ).  
$$
There is a natural identification 
$(\mathbb{Q}/\mathbb{Z})_{|S|} = \mathbb{Q}^{|S|} \cap (-1,0]^ {|S|}$. We will equip $(\mathbb{Q}/\mathbb{Z})_{|S|}$ with a  total order, which we denote $\leq^ !$,  that differs from the restriction of the usual ordering on the rational numbers. 

First of all we define the order $\leq^!$ on $\bZ_{n!}$ recursively, as follows.
\begin{itemize}
\item On $\bZ_{2!}=\{-\frac{1}{2},0\}$ we set $-\frac{1}{2}<^! 0$.
\item Having defined $\leq^!$ on $\bZ_{(n-1)!}$, let us consider the natural short exact sequence
$$
0\to \bZ_{(n-1)!}\to \bZ_{n!}\xrightarrow{\pi_n} \bZ_n\to 0,
$$
where $\bZ_n=\{-\frac{n-1}{n},\hdots, -\frac{1}{n},0\}$ is equipped with the standard order $\leq$ described above. Given two elements $a,b\in \bZ_{n!}$, we set $a\leq^! b$ if either $\pi_n(a)<\pi_n(b)$, or $\pi_n(a)=\pi_n(b)$ and $a \leq^! b$ in $\bZ_{(n-1)!}$, where we are identifying the fiber $\pi_n^{-1}(\pi_n(a))\subseteq \bZ_{n!}$ with $\bZ_{(n-1)!}$ in the canonical manner.
\end{itemize}
We identify $\bZ_{n!}$ with a subset of $\mathbb{Q} \cap (-1, 0]$ as explained in  (\ref{eqpreorder1}) above. 
Note that we can write every element $\chi$ in $\mathbb{Q}^{|S|} \cap (-1,0]^ {|S|}$ as 
$$
\chi = \left(-\frac{p_1}{n!}, \ldots, -\frac{p_N}{n!}\right)
$$
 for some $p_1, \ldots, p_N$ in $\mathbb{N}$ where $n \in \mathbb{N}.$ This expression is unique if we require $n$ to be as small as possible, and we call this the
\emph{normal factorial form}.  
\begin{definition}
\label{defpartord!}
Let 
$$
\chi = \left(-\frac{p_1}{n!}, \ldots, -\frac{p_N}{n!}\right), \quad \chi'= \left(-\frac{q_1}{m!}, \ldots, -\frac{q_N}{m!}\right) \in \mathbb{Q}^  {|S|} \cap (-1,0]^ {|S|}
$$ be in normal factorial form. 
We write $\chi \leq^! \chi'$ if:
\begin{itemize}
\item $n > m$, or
\item $n=m$ and $-\frac{p_i}{n!} \leq^! -\frac{q_i}{n!}$ in $\bZ_{n!}$ for all $i = 1, \hdots, N$, where $\leq^!$ is the ordering defined above. 
\end{itemize}
\end{definition}
\begin{theorem}
\label{mainsgncdivinf}
 \mbox{ }
\begin{enumerate}[leftmargin=*] 
\item The category $\Perf(\radice[\infty]{(X,D)})$ has a  psod
 of type $(\cS(D), \leq)$
  $$ 
\Perf(\radice[\infty]{(X,D)}) = \langle \cA_S, S \in \cS(D) \rangle,
$$
where: 
\begin{itemize}[leftmargin=*]
\item $\cA_{X}\simeq \Perf(X)$.
\item  For all $S \in  \cS(D)^*$, the subcategory $\cA_S$  has a psod of type $((\mathbb{Q}/\mathbb{Z})_{|S|}^*, \leq^!)$ 
$$
\cA_S = \langle  \cA_{ \chi}^S, \chi \in ((\mathbb{Q}/\mathbb{Z})_{|S|}^*, \leq^!)  \rangle
$$ 
and, for all $\chi \in \mathbb{Z}_{|S|, r}^*$, there is an equivalence $\cA_{ \chi}^S \simeq \Perf(\widetilde{S})$. 
\end{itemize}
\item The  category $\Perf(\radice[\infty]{(X,D)})$ has a psod of type $(
(\mathbb{Q}/\mathbb{Z})_{\cS(D)}, \leq^!)$ 
$$\Perf(\radice[\infty]{(X,D)})=  \langle  \cA_\chi^S, \chi \in  ((\mathbb{Q}/\mathbb{Z})_{\cS(D)}, \leq^!)  \rangle,$$
where $\cA^{X} \simeq \Perf(X)$ and for all $\chi \in (\mathbb{Q}/\mathbb{Z})_{|S|}^*$ there is an equivalence $\cA_{ \chi}^S \simeq \Perf(\widetilde{S}).$
\end{enumerate}
\end{theorem} 

\subsection{Kummer flat K-theory of normal crossing divisors}
\label{kfkoncd}
The 
psod-s on infinite root stacks that we constructed in \cite{SST2} are a categorification of structure theorems in Kummer flat K-theory of \emph{simple} normal crossing divisors  due to Hagihara and Nizio\l   \, \cite{hagihara}, \cite{Ni1}. Our techniques  allowed us to extend those structure theorems to a wider class of log stacks, including the case of general normal crossing divisors, but only over a field of characteristic zero and via a desingularization step. 

In this section we explain how the results obtained in Section \ref{secirsgenncd} yield unconditional structure theorem for Kummer flat K-theory of general normal crossing divisors. As in \cite{SST2}, we will formulate our statements more generally  in terms of \emph{noncommutative motives} and \emph{additive invariants} of dg-categories. 


\begin{definition}
An \emph{additive invariant} of dg-categories is a functor of $\infty$-categories 
$$
\mathrm{H}\colon \dgCat \longrightarrow \cP 
$$
where $\cP$ is a stable and presentable 
$\infty$-category, satisfying the following properties:
\begin{enumerate} 
\item $\mathrm{H}$ preserves zero-objects.
\item $\mathrm{H}$ sends split exact sequences of dg-categories to cofiber sequences in $\cP$.
\item $\mathrm{H}$ preserves filtered colimits.
\end{enumerate}
\end{definition}
Hochschild homology, algebraic K-theory and non-connective K-theory are all examples of additive invariants. 
As proved in \cite{tabuada2008}, \cite{BGT} (see also \cite{robalo2015k} and \cite{HSS}) there exists a \emph{universal} additive invariant
$$
\cU\colon \dgCat \longrightarrow \mathrm{Mot}.
$$
The target category of the universal additive invariant, $\mathrm{Mot}$, is called the category of \emph{additive noncommutative motives}. 
If $X$ is a  stack, we denote $\cU(\Perf(X))$ simply by $\cU(X)$. 

The following Corollary extends to the general normal crossing case Corollary 5.6 and 5.8 of \cite{SST2}; its second half extends to the general normal crossing case  Theorem 1.1 of \cite{Ni1}.  
\begin{corollary}
\label{ncmotdirectsum}
Let $(X, D)$ be a log stack given by an algebraic stack $X$ equipped with a normal crossing divisor 
$D.$ 
\begin{itemize}[leftmargin=*]
\item 
There is an equivalence 
\begin{equation}
\label{directdirsum}
\cU(\Perf(\radice[\infty]{(X,D)})) \simeq \cU(X) \bigoplus \Big ( \bigoplus_{S \in \cS(D)^*} \Big ( \bigoplus_{\chi \in (\mathbb{Q}/\mathbb{Z})_S^*} \cU(\widetilde{S}) \Big ) \Big ).  
\end{equation}
Since $\cU$ is universal, every additive invariant $\mathrm{H}(\Perf(\radice[\infty]{(X,D)}))$ decomposes as a direct sum patterned after (\ref{directdirsum}).  
\item Denote by $K_{\mathrm{Kfl}}(X, D)$  the Kummer flat K-theory of $(X,D)$. Then  there is a direct sum decomposition of spectra 
$$ 
K_{\mathrm{Kfl}}(X, D) \simeq K(X) \bigoplus \Big ( \bigoplus_{S \in S_D^*} \Big ( \bigoplus_{\chi \in (\mathbb{Q}/\mathbb{Z})_S^*} K(\widetilde S) \Big ) \Big ).
$$
\end{itemize}
\end{corollary}
\begin{proof}
The first part of the Corollary is proved exactly as Corollary 5.6 from \cite{SST2}. As explained in Section 2.1.4 of \cite{SST2}, it follows from \cite{TV} that the Kummer flat K-theory of $(X,D)$ coincides with the algebraic K-theory of   $\Perf(\radice[\infty]{(X,D)})$. Since K-theory is an additive invariant, the second half of the Corollary follows from the first. This concludes the proof. \end{proof}

\begin{remark}
In addition to Kummer flat K-theory one can define Kummer flat versions of all additive invariants, such as Hochschild homology, in the following way. 
If $X$ is a log scheme, let $\Perf(X_{\mathrm{Kfl}})$ be the dg-category of perfect complexes over the Kummer flat topos of $X$. Then for every additive invariant H, we set $\mathrm{H}_{\mathrm{Kfl}}(X):=\mathrm{H}(\Perf(X_{\mathrm{Kfl}}))$. Corollary \ref{ncmotdirectsum}  implies 
that, if the log structure on $X$ is given by a normal crossing divisor, 
$\mathrm{H}_{\mathrm{Kfl}}(X)$ also decomposes as a direct sum patterned after (\ref{directdirsum}). 
\end{remark}

\begin{remark}
\label{Kumetale}
Corollary \ref{ncmotdirectsum} has an analogue for the \emph{Kummer \'etale} topos of $(X,D)$: this, in particular, extends the second part of the statement of Theorem 1.1 of \cite{Ni1} and the Main Theorem of \cite{hagihara} to the general normal crossing setting. In characteristic zero there is no difference so this comment is relevant only if $\kappa$ has positive or mixed characteristic, and, assuming that $D$ is equicharacteristic as in  \cite{Ni1}, $\bQ/\bZ$ has to be replaced by $(\bQ/\bZ)'=\bZ_{(p)}/\bZ$ (where $p$ is the characteristic over which $D$ lives) in the formulas above. The key  observation is that, if 
$X_{\mathrm{K\acute{e}t}}$ is the Kummer \'etale topos, then $\Perf(X_{\mathrm{K\acute{e}t}})$   is equivalent to perfect complexes over a \emph{restricted version} of the infinite root stack $\radice[\infty']{(X,D)}$, where we take the inverse limit only of root stacks $\radice[r]{(X,D)}$ such that $p$ does not divide $r$. Then $\Perf(\radice[\infty']{(X,D)})$ carries a psod which analogous to the one given by Theorem \ref{mainsgncdivinf}, except we work everywhere with indices which are coprime to $p$. We leave the details to the interested reader.
\end{remark}

\bibliographystyle{plain}
\bibliography{biblio}

\end{document}